\documentclass[a4paper, 12pt]{article}

\usepackage[sort&compress]{natbib}
\bibpunct{(}{)}{;}{a}{}{,} 

\usepackage{amsthm, amsmath, amssymb, mathrsfs, multirow, url, subfigure}
\usepackage{graphicx} 
\usepackage{ifthen} 
\usepackage{amsfonts}
\usepackage[usenames]{color}
\usepackage{fullpage}
\usepackage[normalem]{ulem}
\usepackage{tcolorbox}
\usepackage{accents}
\usepackage{soul}
\usepackage[ruled,vlined]{algorithm2e}
\usepackage{soul}

\theoremstyle{plain} 
\newtheorem{thm}{Theorem}

\newtheorem{prop}{Proposition}
\newtheorem{lem}{Lemma}
\newtheorem*{conj}{Claim}

\theoremstyle{definition}
\newtheorem{defn}{Definition}

\theoremstyle{remark} 

\newtheorem*{astep}{A--step}
\newtheorem*{pstep}{P--step}
\newtheorem*{cstep}{C--step}

\newcommand{\prob}{\mathsf{P}} 
\newcommand{\E}{\mathsf{E}}

\newcommand{\qrob}{\mathsf{Q}}
\newcommand{\Q}{\mathscr{Q}}

\newcommand{\lPi}{\underline{\Pi}}
\newcommand{\uPi}{\overline{\Pi}}

\newcommand{\ber}{{\sf Ber}}
\newcommand{\unif}{{\sf Unif}}
\newcommand{\nm}{{\sf N}}

\newcommand{\stt}{{\sf t}}

\newcommand{\I}{\mathscr{I}}
\newcommand{\RR}{\mathbb{R}}

\newcommand{\PP}{\mathbb{P}}

\newcommand{\YY}{\mathbb{Y}}

\newcommand{\Y}{\mathscr{Y}}

\renewcommand{\S}{\mathcal{S}}

\newcommand{\plint}{\mathscr{P}}

\newcommand{\eps}{\varepsilon}

\newcommand{\iid}{\overset{\text{\tiny iid}}{\,\sim\,}}

\title{Validity, consonant plausibility measures, and conformal prediction}
\author{Leonardo Cella\footnote{Department of Statistics, North Carolina State University; {\tt lolivei@ncsu.edu}, {\tt rgmarti3@ncsu.edu}} \quad and \quad Ryan Martin$^\dagger$}
\date{\today}

\begin{document}

\maketitle 

\begin{abstract}
Prediction of future observations is an important and challenging problem.  The two mainstream approaches for quantifying prediction uncertainty use prediction regions and predictive distributions, respectively, with the latter believed to be more informative because it can perform other prediction-related tasks.  The standard notion of validity, what we refer to here as {\em Type-1 validity}, focuses on coverage probability of prediction regions, while a notion of validity relevant to the other prediction-related tasks performed by predictive distributions is lacking. Here we present a new notion, called {\em Type-2 validity}, relevant to these other prediction tasks.  We establish connections between Type-2 validity and coherence properties, and show that imprecise probability considerations are required in order to achieve it.  We go on to show that both types of prediction validity can be achieved by interpreting the conformal prediction output as the contour function of a consonant plausibility measure.  We also offer an alternative characterization of conformal prediction, based on a new nonparametric inferential model construction, wherein the appearance of consonance is natural, and prove its validity.  


\smallskip

\end{abstract}

\section{Introduction}
\label{S:intro}

Reliable prediction of future observations under minimal model assumptions is an important and fundamental problem in statistics and machine learning.  There are two general approaches to quantifying uncertainty about the yet-to-be-observed value: one is based on {\em prediction regions} and the other is based on {\em predictive distributions}.  Since predictive distributions can be used to construct prediction regions, they are believed to be more informative.  However, the additional information that comes with a full predictive distribution usually requires model assumptions, which creates several practical challenges.  First, with model assumptions comes the risk of model misspecification biases that can negatively impact the quality of predictions; second, computation often requires Monte Carlo methods, and for sufficiently flexible predictive distributions based on, e.g., modern Bayesian nonparametric models \citep[e.g.,][]{hhmw2010, ghosal.vaart.book}, this can be non-trivial.  In a given application, the data analyst might be willing to take on these challenges, but not solely for the construction of prediction regions.  Indeed, there are non-model-based prediction regions with provably exact/conservative coverage probability, while those derived from a model-based predictive distribution, at best, attain the target coverage probability asymptotically.  Therefore, the motivation for developing a predictive distribution for quantifying uncertainty about the yet-to-be-observed value must be that there are other tasks, such as assigning probabilities (or degrees of belief) to general assertions about the to-be-observed value, which are of higher priority than the construction of prediction regions.  This begs the question: {\em how can the quality of a predictive distribution be assessed with respect to these other tasks?}  

In this paper, we develop a new notion of prediction validity that goes beyond the coverage probability of prediction regions, one that aims to assess the performance of a predictive distribution on these other relevant tasks.  If a predictive distribution is to be used to evaluate the probability of certain assertions about the future observation, then the magnitude of this probability will be used to draw inference about whether the assertion will be true.  So, roughly, to avoid making systematically erroneous predictions, it would be desirable to control the rate at which the predictive distribution assigns small probability to assertions that happen to be true.  Our definition of {\em Type-2 prediction validity} in Section~\ref{SS:strong}---compared to prediction region coverage probabilities, which we refer to as {\em Type-1 prediction validity} in Section~\ref{SS:weak}---makes this notion of avoiding systematically erroneous predictions precise.  There are a number of important imprecise probability-related consequences of Type-2 validity.  First, Type-2 validity rules out the possibility of a sure loss, forging an interesting connection between the classical behavioral interpretation of imprecise probabilities and the statistical properties of a procedure based on them. Second, we strengthen the validity property to what we call {\em strong Type-2 validity} in Section~\ref{SS:type3}, which provides control on erroneous predictions uniformly as opposed to pointwise in assertions about $Y_{n+1}$.  A characterization of strong Type-2 validity is given which, among other things, makes clear that additive/precise predictive probability distributions cannot be strongly Type-2 valid.  Therefore, to achieve strong Type-2 prediction validity, one must use an appropriate imprecise probability for prediction uncertainty quantification.  Compare this to the conclusions from the false confidence theorem \citep{Ryansatellite, MARTIN2019IJAR} in the context of statistical inference.


Fortunately, as we show in Section~\ref{S:consonance}, there is a simple imprecise prediction probability that can achieve both Type-1 and (strong) Type-2 prediction validity.  Specifically, we argue that, if the ``p-value'' or ``transducer'' that results from the conformal prediction procedure \citep[e.g.,][]{Vovk:2005, Shafer2007ATO, balasubramanian2014conformal}, is converted into a consonant plausibility function, then both notions of validity hold.  Thus, the present paper ties together two seemingly distinct threads in Glenn Shafer's distinguished research career, namely, imprecise probability/plausibility functions and conformal prediction.  This also shows that this new notion of prediction validity requires no new technical machinery, only a reinterpretation of the conformal transducer as a plausibility contour and adoption of the (consonant) plausibility or possibility calculus for uncertainty quantification.  Further desirable consequences of this new perspective on conformal prediction are that plausibility functions are coherent in the sense of \citet{walley1991} and \citet{lower.previsions.book}, and can be readily combined with a loss/utility function, as discussed in \citet{huntley.etal.decision}, for formal decision-making.  

A ``conformal + consonance'' perspective/approach is attractive due to its simplicity and ability to achieve both Type-1 and 2 prediction validity, but the addition of consonance might appear like an ad hoc adjustment or an after-thought.  In Section~\ref{S:im} we offer a new characterization of conformal prediction, within the {\em inferential model} (IM) framework of \citet{mainMartin, martinbook}, in which the consonance feature is a direct consequence of the construction.  Indeed, a distinguishing feature of the IM approach is its use of nested and suitably calibrated random sets for quantifying uncertainty about unobservable auxiliary variables.  Since nested random sets and calibration go hand in hand with consonance and validity, respectively, this is an interesting and natural way to interpret conformal prediction.  Beyond this characterization, the particular IM construction employed here---an extension of ideas in \citet{martin2015,MARTIN2018105}---is, to our knowledge, the first that is not based on parametric statistical model assumptions.  This nonparametric IM construction is of independent interest, and we are optimistic that the same ideas can be extended beyond prediction to the inference problem as well.  

Numerical examples are presented in Section~\ref{S:examples} to illustrate the IM-based approach to this prediction problem, including a two-dimensional prediction problem using the notion of {\em data depth} \citep{Tukey1975MathematicsAT, liu1999}.   We conclude in Section~\ref{S:discuss} with a summary and a few remarks about open problems.  The appendix contains a number of additional technical details.

\section{Background}
\label{S:background}

\subsection{The statistical problem}
\label{SS:stat}

To set the scene, suppose that there is an exchangeable process $Y_1,Y_2,\ldots$ with distribution $\prob$, where each $Y_i$ is a random variable, vector, etc.~taking values in a space $\YY$.  Recall that the sequence is {\em exchangeable} if, for any natural number $K$ and for any permutation $\sigma$ of the set $\I_K := \{1,2,\ldots,K\}$ of integers, the two random vectors, 
\[ (Y_1,\ldots,Y_K) \quad \text{and} \quad (Y_{\sigma(1)}, \ldots, Y_{\sigma(K)}), \]
have the same joint distribution.  This, in particular, implies that the marginal distributions of the $Y_i$'s are the same, so the process has no ``trend.'' But exchangeability allows for certain kinds of dependence, so our setup is more general than the common independent and identically distributed (iid) formulation.  Note also that we are not assuming any particular parametric form for the distribution $\prob$.  

The statistical problem is as follows.  Suppose we observe the first $n$ terms of the process, i.e., $Y^n=(Y_1,\ldots,Y_n)$.  With this data, and the assumption of exchangeability, the goal is to predict $Y_{n+1}$ using a method that is {\em valid} or reliable in a certain sense.  Section~\ref{S:validity} below provides details about the validity property.  Here we give a bit of background about the prediction problem at hand in the context of an example.  

For illustration, consider a setting where the $Y_i$'s are believed to be iid, with $\YY = \RR$.  If we furthermore assume that the common distribution is $\nm(\mu, \sigma^2)$, with unknown mean $\mu$ and variance $\sigma^2$, then the textbook $100(1-\alpha)$\% prediction interval for $Y_{n+1}$, based on the observed data $Y^n = y^n$, is 
\[ \plint_\alpha(y^n) = \hat\mu_n \pm t_{n-1}(\alpha/2) \hat\sigma_n (1+n^{-1})^{1/2}, \]
where $\hat\mu_n$ and $\hat\sigma_n$ are the sample mean and standard deviation of $y^n$, respectively, $t_{n-1}(\alpha/2)$ is the upper $\alpha/2$ quantile of the Student-t distribution with $n-1$ degrees of freedom, and $1-\alpha \in [0,1]$ is the user-specified confidence level.  To summarize these prediction intervals across different $\alpha$ levels, one can construct a kind of {\em predictive probability distribution} for $Y_{n+1}$, given $Y^n=y^n$, as 
\[ \Pi_{y^n}(A) = \stt_{n-1}\{A \mid \hat\mu_n, \hat\sigma_n(1+n^{-1})^{1/2}\}, \quad A \subseteq \YY, \]
where $\stt_{n-1}(A \mid m, s)$ denotes the probability of the event ``$m + s T \in A$'' when $T$ is a Student-t random variable with $n-1$ degrees of freedom.  If $\prob = \nm(\mu, \sigma^2)$, then it is well-known that the prediction interval achieves the nominal coverage probability  
\[ \prob\{ \plint_\alpha(Y^n) \ni Y_{n+1} \} = 1-\alpha, \quad \text{for all $(\alpha,n,\mu,\sigma)$}, \]
where the probability is with respect to $(Y^n,Y_{n+1})$; see Definition~\ref{def:weak} below.  One can similarly conclude that the aforementioned predictive distribution, $\Pi_{y^n}$, for $Y_{n+1}$ is reliable in the sense that its quantiles correspond to prediction intervals that achieve the nominal frequentist coverage probability.  

However, if $\prob$ is not normal, then the left-hand side of the above display could be very different from the $1-\alpha$ target.  Since no model assumption is 100\% certain, a notion of reliability that does not require such an assumption is desirable; see Section~\ref{SS:weak} below.  Beyond coverage probability of prediction intervals, the reliability of predictive distribution has received less attention.  Even if normality holds, what reliability guarantees does the predictive distribution offer beyond the coverage probability properties of the prediction intervals derived from it?  Without a model assumption, how might one construct a predictive distribution with certain reliability guarantees?  Is this even possible to do using ordinary probability distributions?  We discuss these questions in Section~\ref{SS:strong}.

\subsection{Imprecise probability}
\label{SS:ip}

Our developments below rely on certain concepts and definitions from the theory of imprecise probability, so here we present the relevant details.  An excellent reference for the material presented in this section is \citet{destercke.dubois.2014}.  

Start by defining a {\em capacity} as a set function $\lPi$ that maps subsets of $\YY$ to $[0,1]$, such that $\lPi(\varnothing)=0$, $\lPi(\YY) = 1$, and $A \subseteq B$ implies $\lPi(A) \leq \lPi(B)$; strictly speaking, when $\YY$ is not a finite space, one also needs $\lPi$ to be continuous from above and below \citep[e.g.,][Sec.~6.8]{vovk.shafer.2014} but we will skip these details here.  Capacities with no additional structure are too complex for practical purposes.  One basic and common requirement is that the capacity $\lPi$ be {\em super-additive}, i.e., 
\[ \lPi(A \cup B) \geq \lPi(A) + \lPi(B), \quad \text{for all disjoint $A$ and $B$}. \]
If $\lPi$ is additive, in the sense that equality holds in the above display for every pair of disjoint sets, and monotone, then it is a finitely-additive probability measure.  By letting $B = A^c$ be the complement of $A$ in the above display, it follows that
\[ \lPi(A) \leq 1-\lPi(A^c), \quad A \subseteq \YY. \]
With this gap between $\lPi(A)$ and $1-\lPi(A^c)$, it makes sense to define a dual function 
\begin{equation}
\label{eq:dual}
\uPi(A) = 1-\lPi(A^c), \quad A \subseteq \YY. 
\end{equation}
It follows from super-additivity that $\lPi(A) \leq \uPi(A)$ for all $A$, so it is common to refer to $\lPi$ and $\uPi$ as {\em lower} and {\em upper probabilities}, respectively.  



An important special case is when the lower probability $\lPi$ is determined by the distribution of a random set \citep[e.g.,][]{molchanov2005, nguyen2006introduction}, 
\[ \lPi(A) = \prob_\Y(\Y \subseteq A), \quad A \subseteq \YY, \]
where $\Y$ is a random subset of $\YY$ with distribution $\prob_\Y$, non-empty with $\prob_\Y$-probability~1.  A further special case, especially important in what follows, corresponds to a nested random set.  Under this setting, the upper probability $\uPi$ is called a {\em possibility measure} \citep[e.g.,][]{dubois.prade.book, hose.hanss.2021}, a simple yet powerful imprecise probability model.  To relate things back to Shafer's developments, a possibility measure is equivalent to a {\em consonant} belief/plausibility function, one where there exists a function $\pi$, mapping $\YY$ to $[0,1]$, such that $\sup_y \pi(y) = 1$ and the upper probability satisfies
\[ \uPi(A) = \sup_{y \in A} \pi(y), \quad A \subseteq \YY. \]
The function $\pi$ is analogous to the density/mass function that characterizes an ordinary probability distribution, and will be referred to below as a {\em plausibility contour}.  Note that $\pi(y)$ is the upper probability assigned to the singleton $\{y\}$, which means that the entire lower and upper probability pair is determined by the plausibility assigned to singletons.  This is a unique feature of the consonance model. 

Regardless of the mathematical form of the imprecise probability model, the literature focuses primarily on a subjective interpretation, \`a la de Finetti.  That is, the lower and upper probabilities are treated as subjective degrees of belief, but with an associated ``action'' or behavioral element to connect it to the real world.  The standard interpretation is that the lower probability $\lPi(A)$ is the the agent's largest buying price for a gamble that pays 1 unit if the event $A$ is realized.  Similarly, the upper probability $\uPi(A)$ is the agent's smallest selling price for a gamble that pays 1 unit if the event $A$ is realized.  Given this gambling interpretation, it makes sense to consider what it takes to ensure that the agent who follows such a policy cannot be made a sure loser.  This {\em no sure loss} property is a slightly weaker form of that generally referred to as {\em coherence}, and there is a rich literature on the theory of coherent lower and upper probabilities (or previsions more generally); see, e.g., \citet{walley1991}, \citet{lower.previsions.book}, and the references therein. Here it will be enough for the reader to keep in mind that avoiding sure loss is essential to the behavioral interpretation of lower and upper probabilities, ensuring that the agent's internal assessment of his uncertainty is not irrational; that is, the agent believes that he cannot be made a sure loser.  Of course, the gambling setup need not be real in order for these notions to be meaningful \citep{shafer.betting}.

An additional dimension that will be relevant to the present discussion is that our lower and upper probabilities will be data-dependent, i.e., our $\lPi_{y^n}$ and $\uPi_{y^n}$ will depend on the previously observed data $y^n$ in a certain way.  And it may not be through specification of an over-arching (imprecise) probability model and formal updating/conditioning rules as described in, e.g., \citet{walley1991}.  That is, we will be considering general maps from data to lower and upper probability pairs and, in addition to wanting these to be meaningful in a behavioral sense for each fixed $y^n$, we also want the predictions drawn from this formulation to be reliable in a statistical sense, i.e., we will be concerned with the sampling distributions of $\lPi_{Y^n}(A)$ and $\uPi_{Y^n}(A)$ as a function of $Y^n$ having distribution $\prob$ for fixed $A \subseteq \YY$.  The discussion of Type-2 validity in Sections~\ref{SS:strong}--\ref{SS:type3} aim at tying the behavioral and statistical reliability notions together.

\section{Prediction validity}
\label{S:validity}

\subsection{Type-1 validity: prediction region coverage}
\label{SS:weak}

Recall that the goal is to predict $Y_{n+1}$ given an observed set of values $Y^n = (Y_1,\ldots,Y_n)$ from the process $Y_1,Y_2,\ldots$ with distribution $\prob$, assumed throughout to be exchangeable.  As discussed in Section~\ref{S:intro}, the logic behind our statistical reasoning fails if predictions are not valid in a certain sense. But what does ``validity'' mean in this context?  A first and relatively weak requirement is that prediction regions---see Section~\ref{SS:stat}---achieve the advertised coverage probabilities. 

\begin{defn}
\label{def:weak}
Let $\{\plint_\alpha(y^n): \alpha \in [0,1]\}$ denote a family of prediction regions for $Y_{n+1}$ based on observed data $Y^n=y^n$.  Then the prediction is {\em Type-1 valid} if the prediction regions achieve the advertised coverage probability uniformly in $n$ and $\prob$, i.e., 
\begin{equation}
\label{eq:weak}
\prob\{\plint_\alpha(Y^n) \ni Y_{n+1}\} \geq 1-\alpha, \quad \text{for all $(\alpha,n,\prob)$}.
\end{equation}
Here and below, the probability is with respect to the joint distribution of $(Y^n,Y_{n+1})$.  We mention this only to be clear that this is marginal coverage of the prediction region, as opposed to conditional coverage for a given $Y^n=y^n$.
\end{defn}

There are many different strategies available for constructing prediction intervals, one of the most common is Bayesian, which is based on a fully-specified probability model for the observables.  In its most general form, the Bayesian approach starts with a prior distribution for $\prob$, which determines a prior predictive distribution and, ultimately, a posterior predictive distribution for $Y_{n+1}$, given $Y^n=y^n$.  Do prediction intervals derived from a Bayesian predictive distribution satisfy \eqref{eq:weak}?  To investigate this, we consider a simple-yet-powerful Bayesian nonparametric approach, one that assumes data are iid with common marginal distribution, to which we assign a Dirichlet process prior \citep{ferguson1973} with base measure $G$---a probability measure on the $Y$ space---and precision parameter $\delta > 0$.  This model is rather flexible and has strong theoretical support; see, e.g., \citet{ghosal2010} and \citet[][Ch.~4]{ghosal.vaart.book}.  With this formulation, it turns out that the posterior predictive distribution, $\Pi_{Y^n}$, for $Y_{n+1}$ is quite simple: it is a mixture of $G$ and the empirical distribution $\PP_n$ based on $Y^n$, i.e., 
\[ \Pi_{Y^n}(A) = \tfrac{\delta}{\delta + n} \, G(A) + \tfrac{n}{\delta + n} \, \PP_n(A), \quad A \subseteq \RR. \]
Then a family of prediction intervals $\plint_\alpha$ can be readily derived by extracting the $\alpha/2$ and $1-\alpha/2$ quantiles from the aforementioned predictive distribution. 
 
For our simulation, we take the Dirichlet process prior with $G=\nm(0,1)$ and $\delta=1$.  We consider three sample sizes, $n \in \{20, 30, 40\}$, along with three data generating distributions: the normal distribution with mean zero and variance 0.5, the standard Cauchy distribution, and the standard skewed normal distribution \citep{azzalini} with skewness parameter equal to 1. For each of these scenarios, 5000 data sets are generated and, from each, a 90\% prediction interval is extracted.  Table~\ref{tab:coverage} reports the estimated coverage probabilities of these prediction intervals.  Note how the coverage can be quite low, even for relatively large sample sizes.  As one might expect, the coverage probability improves as $n$ increases; however, \eqref{eq:weak} effectively involves an infimum over $\prob$, so the method's performance is determined by the ``worst-case distribution,'' which can be no better than the results in the Cauchy distribution column in Table~\ref{tab:coverage}.  Therefore, Type-1 validity in the sense of \eqref{eq:weak} fails.  

\begin{table}[t]
\centering
\begin{tabular}{cc c c}
\hline
$n$ & Normal & Cauchy & Skew Normal \\
\hline
20 & 0.883 & 0.731 & 0.796  \\
30 & 0.916 & 0.821 & 0.847  \\
40 & 0.921 & 0.854 & 0.906\\
\hline
\end{tabular}
\caption{Estimated coverage probabilities of 90\% Bayesian prediction intervals for $Y_{n+1}$, based on a Dirichlet process prior, with data coming from various true distributions.}
\label{tab:coverage}
\end{table}

A take-away message is that a full probability model---even a very flexible nonparametric one---is apparently not flexible enough for the quantiles of its predictive distribution to be valid prediction intervals in the sense of \eqref{eq:weak}. In fact, as we discuss in Section~\ref{S:discuss}, interesting parallels can be drawn to existing results in the imprecise probability literature that strongly suggest probability measures fail to provide validity in the sense of Definition~\ref{def:weak} and, moreover, that consonant plausibility functions (Section~\ref{SS:consonant}) are the unique imprecise probability that succeed.  An even weaker conclusion achieved by a probability measure is that \eqref{eq:weak} is achieved as $n \to \infty$ for certain $\prob$.  But then the method's utility depends on whether the user is willing to assume that their unknown $\prob$ is in that class and that their $n$ exceeds an unknown threshold.  Therefore, in order to achieve even the relatively weak notion of validity in \eqref{eq:weak}, it seems that considerations beyond classical/precise probability theory are required.  We will reach a similar conclusion when we consider new notions of validity below.

\subsection{Type-2 validity: beyond prediction region coverage}
\label{SS:strong}

Beyond prediction regions, some more general uncertainty quantification may be desired.  For example, with a Bayesian posterior predictive distribution for $Y_{n+1}$, there are many things one can do, including plot the density function, evaluate probabilities for arbitrary assertions about $Y_{n+1}$, marginalize to a predictive distribution for a certain feature of $Y_{n+1}$, etc.  In fact, since one can obtain prediction regions without a predictive distribution, the practical motivation for constructing a predictive distribution in the first place is to perform these other tasks.  Therefore, these other tasks must be of primary importance, so a different notion of validity is needed to ensure that they do not result in systematically misleading predictions.  

To formalize this more general approach to prediction, define a {\em probabilistic predictor} as a map $y^n \mapsto (\lPi_{y^n}, \uPi_{y^n})$ that converts data into a pair of data-dependent lower and upper probabilities for $Y_{n+1}$;  we only require that $\lPi_{y^n}$ and $\uPi_{y^n}$ satisfy the properties of a capacity for (almost) all $y^n$, but we will use the lower/upper probability terminology. This includes the case of a precise predictive distribution $\Pi_{y^n}$, like that based on the Dirichlet process prior in the above example, if $\lPi_{y^n} = \uPi_{y^n} = \Pi_{y^n}$.  For technical reasons, we require the functions $y^n \mapsto \lPi_{y^n}(A)$ and $y^n \mapsto \uPi_{y^n}(A)$ to be measurable with respect to the $\sigma$-algebra on which $\prob$ is defined, for all $A$.  

We have opted for the ``probabilistic predictor'' terminology as opposed to ``confidence predictor'' in \citet{Vovk:2005} to highlight the difference between our objectives and theirs.  As discussed above, we are interested in general uncertainty quantification about $Y_{n+1}$, beyond just prediction regions; for example, we anticipate evaluating (lower and upper) probabilities for general assertions about $Y_{n+1}$.  By an assertion about $Y_{n+1}$, we mean a subset of $\YY$ that may or may not contain $Y_{n+1}$; that is, we associate the subset $A \subseteq \YY$ with the assertion ``$Y_{n+1} \in A$.''  Note that assertions $A$ about $Y_{n+1}$ are virtually unlimited, e.g., $A=\{y \in \YY: \phi(y) \in B\}$ where $\phi$ is any relevant feature and $B \subseteq \phi(\YY)$ is any relevant assertion about $\phi(Y_{n+1})$.  Then the function $A \mapsto (\lPi_{Y^n}, \uPi_{Y^n})(A)$ provides uncertainty quantification about $Y_{n+1} \in A$, given data $Y^n=y^n$.  To assess whether this uncertainty quantification is meaningful, below we consider properties of the sampling distribution of $(\lPi_{Y^n}, \uPi_{Y^n})(A)$ as a function of the data $Y^n$ for fixed $A$. 

For a given assertion $A$ about $Y_{n+1}$, the events,
\begin{equation}
\label{eq:events}
\{\text{$\uPi_{Y^n}(A)$ is small, $Y_{n+1} \in A$}\} \quad \text{and} \quad \{\text{$\lPi_{Y^n}(A)$ is large, $Y_{n+1} \not\in A$}\}, 
\end{equation}
correspond to potentially erroneous predictions and, therefore, are undesirable.  We focus on small and large (upper and lower) probabilities because of what Shafer calls {\em Cournot's principle} \citep{cournot1843}, which states that probability theory is relevant to the real-world only through the assertions it assigns small or large probability to; see \citet{shafer2007} and \citet[][Ch.~10.2]{shafer.vovk.book.2019}.  From this perspective, if a user adopts $(\lPi_{Y^n},\uPi_{Y^n})$ as his predictive distribution, then he is inclined to reject (sell gambles on) those assertions $A$ with sufficiently small $\uPi_{Y^n}(A)$ and accept (buy gambles on) those with sufficiently large $\lPi_{Y^n}(A)$.  And for those two cases, respectively, if $Y_{n+1}$ happens to be in or out of $A$, then his prediction would be wrong (and he loses money).  Therefore, the goal of this new validity property is to ensure that these undesirable events are controllably rare.  

\begin{defn}
\label{def:strong}
The probabilistic predictor $y^n \mapsto (\lPi_{y^n}, \uPi_{y^n})$ is {\em Type-2 valid} if 
\begin{equation}
\label{eq:strong}
\prob\{ \uPi_{Y^n}(A) \leq \alpha \, , \, Y_{n+1} \in A\} \leq \alpha, \quad \text{for all $(\alpha,n,A,\prob)$}.
\end{equation}
\end{defn}


The ``for all $A$'' part of the definition is important for at least three reasons.  First, as we discussed above, the motivation for introducing a predictive distribution in the first place is to be able to evaluate probabilities for general assertions about $Y_{n+1}$, so limiting the reliability of these probabilities to certain ``nice'' assertions would defeat the purpose. Second, since 
\[ \prob\{ \uPi_{Y^n}(A) \leq \alpha \, , \, Y_{n+1} \in A\} \leq \prob(Y_{n+1} \in A), \]
one could easily find a single $A$ to make the upper bound less than $\alpha$ independent of how the probabilistic predictor assigns belief.  Third, if the bound \eqref{eq:strong} holds for all $A$, then a similar bound holds for the second kind of event in \eqref{eq:events}.  To see this, by the duality property in \eqref{eq:dual}, i.e., $\lPi_{Y^n}(A) = 1-\uPi_{Y^n}(A^c)$, it follows that 
\[ \prob\{\lPi_{Y^n}(A) \geq 1-\alpha \, , \, Y_{n+1} \not\in A\} = \prob\{\uPi_{Y^n}(A^c) \leq \alpha \, , \, Y_{n+1} \in A^c\}. \]
But $A^c$ is an assertion too, so it follows from the ``for all $A$'' part of \eqref{eq:strong} that the right-hand side is bounded by $\alpha$, hence the second event in \eqref{eq:events} has controllably small probability.  

One might argue that forcing the probabilistic predictor to satisfy these calibration properties for all assertions is too restrictive.  In any fixed application, the data analyst is free to decide which assertions are deserving of focus and which can be ignored, and he must accept the consequences of his decisions, positive or negative.  However, as developers of general methods for use in all sorts of applications, statisticians should not be making those decisions for the data analysts.  Instead, the methods we develop should be as conservative as necessary to provide adequate protection against the most challenging question a practitioner might ask.  One can also view this from the behavioral  perspective described in Section~\ref{SS:ip}.  Once an agent advertises/commits to his buying and selling prices based on the probabilistic predictor, the choice of $A$ is no longer in his hands---competing agents will strategically choose $A$'s to maximize their winnings.  So, in constructing a probabilistic predictor, all relevant assertions should be considered.  Besides, there is an even stronger version of Type-2 validity that can be established in certain cases; see Section~\ref{SS:type3}. 


Some further insights can be gained about what it takes to achieve \eqref{eq:strong} by re-expressing the probability on the left-hand side.  By conditioning on $Y^n$ and using the iterated expectation formula, we get 
\begin{align*}
\prob\{ \uPi_{Y^n}(A) \leq \alpha \, , \, Y_{n+1} \in A\} & = \E \bigl( 1_{\uPi_{Y^n}(A)\leq \alpha} \, 1_{Y_{n+1} \in A} \bigr) \\
& = \E \bigl\{ \E\bigl( 1_{\uPi_{Y^n}(A)\leq \alpha} \, 1_{Y_{n+1} \in A} \mid Y^n \bigr) \bigr\} \\
& = \E\bigl\{ 1_{\uPi_{Y^n}(A)\leq \alpha} \, \prob(Y_{n+1} \in A \mid Y^n) \bigr\},
\end{align*}
where $1_E$ denotes the indicator of event $E$. Then \eqref{eq:strong} implies 
\begin{equation}
\label{eq:strong.alt}
\E\bigl\{ 1_{\uPi_{Y^n}(A)\leq \alpha} \, \prob(Y_{n+1} \in A \mid Y^n) \bigr\} \leq \alpha \quad \text{for all $(\alpha,n,A,\prob)$}. 
\end{equation}
The inequality is clearly satisfied by the true conditional probability, i.e., $\uPi_{Y^n}(\cdot) \equiv \prob(Y_{n+1} \in \cdot \mid Y^n)$, but are there any other probabilities that satisfy it?  A closer look at this inequality reveals that the probabilistic predictor must {\em dominate} the unknown conditional distribution in a certain weak sense. Roughly, for those $Y^n$ such that $\prob(Y_{n+1} \in A \mid Y^n)$ tends to be large, $\uPi_{Y^n}(A)$ cannot tend to be small. Compare this to the prediction calibration property in Equation~(3) of \citet{denoeux.li.2018}---see, also, \citet{denoeux2006}---that, in our notation/terminology, aims to control the $\prob$-probability that $\uPi_{Y^n}(A)$ is bigger than $\prob(Y_{n+1} \in A \mid Y^n)$.  Our focus, on the other hand, is on directly controlling the probability of those ``bad'' events in \eqref{eq:events} that may lead to prediction errors.  

Dominance properties are common in the imprecise probability literature, so it is natural to compare these more familiar notions of dominance with that implied by Type-2 validity.  This comparison is interesting because coherence---in the sense of de Finetti, Walley, Williams, etc.---or the slightly weaker no sure loss property is focused on the internal rationality of probabilistic reasoning, while validity is focused on the external operating characteristics of a user-specified probabilistic predictor.  Intuitively, a procedure cannot be effective if it violates certain logical principles but formal investigations into this connection are limited.  In this direction, we have the following basic result.

\begin{prop}
\label{prop:no.sure.loss}
If, for some assertion $A$ about $Y_{n+1}$, with $A \subseteq \YY$, the probabilistic predictor is strictly upper-bounded away from the true marginal probability, i.e., if  
\begin{equation}
\label{eq:sure.loss}
\sup_{y^n} \uPi_{y^n}(A) < \prob(Y_{n+1} \in A), 
\end{equation}
then Type-2 validity in the sense of Definition~\ref{def:strong} fails.  
\end{prop}

\begin{proof}
Since ``$\sup_{y^n} \uPi_{y^n}(A) \leq \alpha$'' implies ``$\uPi_{Y^n}(A)\leq \alpha$,'' the indicator function of the former event is no more than that of the latter.  Consequently, the left-hand side of \eqref{eq:strong.alt} can be lower bounded as 
\[ \E\bigl\{ 1_{\uPi_{Y^n}(A)\leq \alpha} \, \prob(Y_{n+1} \in A \mid Y^n) \bigr\} \geq 1_{\sup_{y^n} \uPi_{y^n}(A) \leq \alpha} \, \prob(Y_{n+1} \in A). \]
By \eqref{eq:sure.loss}, there exists a number $\alpha$ such that 
\[ \sup_{y^n} \uPi_{y^n}(A) < \alpha < \prob(Y_{n+1} \in A). \]
It follows that, for this $\alpha$, 
\[ \E\bigl\{ 1_{\uPi_{Y^n}(A)\leq \alpha} \, \prob(Y_{n+1} \in A \mid Y^n) \bigr\} \geq \prob(Y_{n+1} \in A) > \alpha, \]
which is a violation of the inequality in \eqref{eq:strong.alt}.  
\end{proof}

The hypothesis \eqref{eq:sure.loss} of the above proposition is an instance of sure loss as discussed in the imprecise probability literature; see, e.g., Condition~(C7) in \citet[][Sec.~6.5.2]{walley1991} or Definition~3.3 in \citet{gong.meng.update}.  Since sure loss is among the most egregious violations of rationality that the theory strives to avoid, we find comfort in the {\em validity implies no sure loss} conclusion that can be drawn from Proposition~\ref{prop:no.sure.loss}.  Note that the same conclusion can be reached if the probabilistic predictor's lower probabilities are strictly and uniformly lower-bounded away from the marginal probability.  

Continuing in this direction, the no sure loss phenomenon is like a very basic behavioral sanity check, ensuring that the agent who sets buying and selling prices based on a Type-2 valid probabilistic predictor is not sure to lose money.  The coherence property mentioned in passing in Section~\ref{SS:ip} is more fundamental than no sure loss, so it makes sense to consider how Type-2 validity and coherence might be related.  It turns out that there is a close connection between coherence of an upper probability and the set of probability measures it dominates.  In particular, \citet[][Theorem~2]{williams.previsions} proved a version of the celebrated {\em upper envelope theorem}, which states that a conditional upper probability is (W-)coherent if and only if it equals the upper envelope of a suitable collection of ordinary conditional probabilities.  The notion of coherence in \citet[][Sec.~7.1.4]{walley1991} is generally stronger than W-coherence---``W'' for Williams---so Walley's envelope theorem (Sec.~7.1.6) has only the ``if'' part of Williams's; see \citet[][Appendix~K]{walley1991} for details.  This distinction is not important for us here because the notion of dominance implied by \eqref{eq:strong.alt} is relative to a particular conditional probability measure; so being an upper envelope of {\em some} collection of ordinary conditional probabilities is not enough to forge a connection between Type-2 validity and coherence.  Instead, we have to be explicit about the collection of probabilities.  

Let $\Q$ denote a collection of candidate joint distributions $\qrob$ for the process $Y_1,Y_2,\ldots$.
Given $Y^n=y^n$, under the appropriate regularity conditions on $\Q$ (weak-$*$ compact and convex), the upper envelope \eqref{eq:envelope} of the conditional probabilities corresponding to members of $\Q$ is a coherent upper probability in the sense of Walley. Moreover, if $\Q$ contains the true distribution $\prob$, then the upper envelope also satisfies Type-2 validity.  

\begin{prop}
\label{prop:coherent}
Given the collection of probability distributions $\Q$, and given $Y^n=y^n$, define the probabilistic predictor 
\begin{equation}
\label{eq:envelope}
\uPi_{y^n}(A) = \sup\{\qrob(Y_{n+1} \in A \mid y^n): \qrob \in \Q\}.
\end{equation}
If $\Q$ contains the true distribution $\prob$ of the process, then the probabilistic predictor above is Type-2 valid in the sense of Definition~\ref{def:strong}. 
\end{prop}

\begin{proof}
By definition of the probabilistic predictor, if $\prob \in \Q$, then 
\[ \uPi_{y^n}(A) \geq \prob(Y_{n+1} \in A \mid y^n), \]
and the claim follows immediately from \eqref{eq:strong.alt}.  
\end{proof}

It is not difficult to define a collection $\Q$ of candidate joint distributions, the challenge is being comfortable with the non-trivial assumption that $\prob \in \Q$.  For one extreme example, suppose that $\prob$ is actually known.  Then, as mentioned above, we could take $\uPi_{y^n}(A)$ equal to $\prob(Y_{n+1} \in A \mid y^n)$ and achieve both validity and coherence, in the most efficient way possible.  For the other extreme, suppose that nothing beyond exchangeability of $\prob$ is known, which is the situation we are considering in this paper.  Then $\Q$ is huge and the upper envelope to the collection of corresponding conditional distributions is effectively vacuous.  The vacuous upper probability is both valid and coherent, but in the least efficient way possible, rendering it effectively useless.

For situations like in the present paper, where virtually nothing about $\prob$ is known but we want efficient predictions, we cannot rely on a single probability model or on the upper envelope derived from a collection $\Q$ of such models, so a new idea is needed.  Fortunately, it is possible to construct a  probabilistic predictor that is valid and only requires an assumption of exchangeability in $\prob$; see Section~\ref{S:consonance}.  This approach requires imprecision in the probabilistic predictor, but does not proceed by formulating an imprecise probability model, $\Q$, and following formal updating rules.  In fact, the notion of validity that this approach achieves, discussed next, is even stronger than Type-2 validity.

\ifthenelse{1=1}{}{
While Proposition~\ref{prop:coherent} gives only a sufficient condition for validity, it is clear that validity cannot be achieved in a practically useful way via a model-based approach as described above without making some additional structural assumptions about $\prob$.  Moreover, the above discussion strongly suggests the following 

\begin{conj}
If the probabilistic predictor is an additive probability measure and satisfies Type-2 validity, then it must depend on $\prob$ and, {\color{red}therefore, is inaccessible to practitioners.}
\end{conj}


We have not been able to prove this claim rigorously, 
but we do provide some additional heuristic support for this claim in Appendix~\ref{S:conjecture}. 
}

\subsection{Strong Type-2 validity: uniformity in $A$}
\label{SS:type3}

Type~2 validity's ``for all $A$'' component is critical for interpretation, as we explained above, but to some this may not be strong enough.  Indeed, the inequalities in \eqref{eq:strong} and \eqref{eq:strong.alt} hold assertion-wise in $A$, which may not be fully satisfactory.  In the basic gambling context we considered above, the agent's opponents would have access to the data at the time of prediction and, therefore, could make a strategic, perhaps data-dependent, choice of assertion $A$ for their transaction with the agent.  Having uniform-in-assertions control on the probability of making erroneous predictions would, therefore, be desirable to the agent since it helps to protect against these strategic choices.  Here we introduce a {\em strong Type-2 validity} property that provides this uniform control and discuss its consequences.  This idea developed out of personal communications with Professor V.~Vovk.

\begin{defn}
\label{def:type3}
The probabilistic predictor $y^n \mapsto (\lPi_{y^n}, \uPi_{y^n})$ is {\em strongly Type-2 valid} if 
\begin{equation}
\label{eq:type3}
\prob\{\text{$\uPi_{Y^n}(A) \leq \alpha$ and $Y_{n+1} \in A$ for some $A$}\} \leq \alpha, \quad \text{for all $(\alpha, n, \prob)$}. 
\end{equation}
(That the event in the above probability statement is measurable---despite being an uncountable union in general---is a consequence of Proposition~\ref{prop:type3} below.)
\end{defn}

Using the duality \eqref{eq:dual}, it is easy to check that an analogous condition holds for the lower probability $\lPi_{Y^n}$.  That is, \eqref{eq:type3} is equivalent to 
\[ \prob\{\text{$\lPi_{Y^n}(A) \geq 1-\alpha$ and $Y_{n+1} \not\in A$ for some $A$}\} \leq \alpha, \quad \text{for all $(\alpha, n, \prob)$}. \]
Thus, strong Type-2 validity controls the $\prob$-probability of both undesirable events in \eqref{eq:events}, uniformly in assertions. Compare this to the assertion-wise control in \eqref{eq:strong}.

As the name suggestions, strong Type-2 validity is stronger than Type-2: the left-hand side of \eqref{eq:type3} is no smaller than $\prob\{\uPi_{Y^n}(A) \leq \alpha \, , Y_{n+1} \in A\}$ for every $A$.  Therefore, it follows from Proposition~\ref{prop:no.sure.loss} above that strong Type-2 validity implies the probabilistic predictor is safe from the undesirable sure loss property \eqref{eq:sure.loss}.  But what kind of probabilistic predictor satisfies strong Type-2 validity?  

\begin{prop}
\label{prop:type3}
A necessary and sufficient condition for strong Type-2 validity is 
\begin{equation}
\label{eq:type3.alt}
\prob\bigl\{ \uPi_{Y^n}(\{Y_{n+1}\}) \leq \alpha \bigr\} \leq \alpha \quad \text{for all $(\alpha, n, \prob)$}. 
\end{equation}
\end{prop}

\begin{proof}
We prove that the two events, subsets of $\YY^{n+1}$,  
\begin{align*}
E_1 & = \{y^{n+1}: \text{$\uPi_{y^n}(A) \leq \alpha$ and $y_{n+1} \in A$ for some $A$}\} \\
E_2 & = \{y^{n+1}: \uPi_{y^n}(\{y_{n+1}\}) \leq \alpha\}
\end{align*}
are the same, which we do by proving $E_1 \supseteq E_2$ and $E_1 \subseteq E_2$.  First it is easy to see that $E_1 \supseteq E_2$ since, if $y^{n+1} \in E_2$, then $A$ can be taken as $A=\{y_{n+1}\}$.  Next, to show that $E_1 \subseteq E_2$, recall that the upper probability is monotone: if $A \subseteq B$, then $\uPi_{y^n}(A) \leq \uPi_{y^n}(B)$ for all $y^n$.  If $y^{n+1} \in E_1$, then there is a set $A$ such that $\uPi_{y^n}(A) \leq \alpha$ and contains $y_{n+1} \in A$.  By monotonicity, it follows that $\uPi_{y^n}(\{y_{n+1}\}) \leq \alpha$; therefore, $y^{n+1} \in E_2$ and, hence, $E_1 \subseteq E_2$.  This implies $\prob(E_1) = \prob(E_2)$, from which we can conclude that \eqref{eq:type3} holds if and only if \eqref{eq:type3.alt} holds, hence the claim.   
\end{proof}

In Proposition~\ref{prop:coherent} in Section~\ref{SS:strong} above, we showed that true conditional probability, or any upper envelope corresponding to a set of probabilities that contains $\prob$, is Type-2 valid.  In light of the characterization in Proposition~\ref{prop:type3}, it is clear that, at least in continuous-data problems, the probabilistic predictor cannot be additive and achieve strong Type-2 validity.  This is because the probability on the left-hand side of \eqref{eq:type3.alt} is always equal to 1.  The same is true for the upper envelope.  So, as expected, as we ask for more control on the output of the probabilistic predictor, additivity eventually becomes too restrictive and imprecision is necessary.  

Proposition~\ref{prop:type3} also sheds light on what features of the probabilistic predictor are most relevant.  Indeed, this characterization is in terms of the upper probability assigned to singleton sets, so it is only natural to use a probabilistic predictor that itself is characterized by its upper probability on singletons.  It turns out that there is such a probability model, namely, the consonant plausibility functions mentioned briefly in Section~\ref{SS:ip} and which play an important role in what follows.

\section{Achieving Type-2 validity, I: conformal prediction plus consonance}
\label{S:consonance}

\subsection{Conformal prediction}
\label{SS:conformal}

Here we give a brief introduction to conformal prediction, following the presentations in \citet{Vovk:2005} and \citet{Shafer2007ATO}.  Given $Y^{n+1} = (Y^n, Y_{n+1})$ consisting of the observable $Y^n$ and the yet-to-be-observed $Y_{n+1}$ value, consider the transformation $Y^{n+1} \to T^{n+1}$ defined by the rule 
\begin{equation}
\label{eq:rule}
T_i = \psi_i(Y^{n+1}), \quad i \in \I_{n+1},
\end{equation}
where $\psi_1,\ldots,\psi_{n+1}$ are given by 
\[ \psi_i(y^{n+1}) = \Psi(y_{-i}^{n+1}, y_i), \quad i \in \I_{n+1}, \]
with $y_{-i}^{n+1} = y^{n+1} \setminus \{y_i\}$ and $\Psi: \RR^n \times \RR \to \RR$ a fixed function that is invariant to permutations in its first vector argument.  
The function $\Psi$ is called a {\em non-conformity measure}, and the interpretation is that $\psi_i(y^{n+1})$ is small if and only if $y_i$ agrees with a prediction derived based on the data $y_{-i}^{n+1}$.  The key references above give numerous examples of non-conformity measures; see, also, \citet{leipredictionset} and \citet{hong_martin_2019}.  The basic idea is to define $\psi_i(y^{n+1})$ in such a way that it compares $y_i$ to a suitable summary of $y_{-i}^{n+1}$, for example 
\begin{equation}
\label{eq:mean}
\psi_i(y^{n+1}) = |\text{average}(y_{-i}^{n+1}) - y_i|, \quad i \in \I_{n+1}. 
\end{equation}
When the data consists of a response and covariate pair, similar but more complicated non-conformity measures are often used \citep[e.g.,][]{Shafer2007ATO, wassermanregressionprediction}.  The essential feature is that the mapping $Y^{n+1} \to T^{n+1}$ preserves exchangeability.  

Of course, since the goal is to predict $Y_{n+1}$, its value is not observed, so the above calculations cannot exactly be carried out.  However, the exchangeability-preserving properties of the transformations described above provide a procedure to suitably rank candidate values $\tilde y$ of $Y_{n+1}$ based on the observed $Y^n=y^n$; see Algorithm~\ref{algo:conformal}.  The output of this algorithm is a data-dependent function $\tilde y \mapsto \pi(\tilde y; y^n)$ whose interpretation is as a measure of how plausible is the claim ``$Y_{n+1}=\tilde y$'' based on data $y^n$.  In \citet{Vovk:2005}, this function is referred to as a ``p-value'' (p.~25) or as a ``conformal transducer'' (p.~44), but we prefer the name {\em plausibility contour} for various reasons, one being that its interpretation is clear.  One important role the plausibility contour plays is in the construction of conformal prediction regions.  Indeed, the family of sets defined by 
\begin{equation} 
\label{eq:conformal.plint}
\plint_\alpha(y^n) = \{y_{n+1} \in \YY: \pi(y_{n+1}; y^n) > \alpha\},
\end{equation}
satisfies the prediction coverage probability property \eqref{eq:weak} and, therefore, conformal prediction is Type-1 valid.  This is what \citet{Vovk:2005} call ``(conservatively) valid'' (p.~20).  We show below that there is more that can be done with the plausibility contour. 

\begin{algorithm}[t]
\SetAlgoLined
 initialize: data $y^n$, non-conformity measure $\Psi$, and a grid of $\tilde y$ values\;
 \For{each $\tilde y$ value on the grid}{
  set $y_{n+1}=\tilde y$ and write $y^{n+1} = y^n\cup\{y_{n+1}\}$ \;
  define $T_i = \psi_i(y^{n+1})$ for each $i \in \I_{n+1}$\;
  evaluate $\pi(\tilde y; y^n) = (n+1)^{-1} \sum_{i=1}^{n+1} 1_{T_i \geq T_{n+1}}$\;
 }
 return $\pi(\tilde y; y^n)$ for each $\tilde y$ on the grid.
 \caption{\textbf{Conformal Prediction}}
 \label{algo:conformal}
\end{algorithm}

\subsection{Consonance and Type-2 validity}
\label{SS:consonant}

Most, if not all, proofs of Type-1 prediction validity in the literature are based on the identification of a function $\pi$, taking values in $[0,1]$, such that 
\begin{equation}
\label{eq:pi.1}
\text{$\pi(Y_{n+1}; Y^n)$ is stochastically no smaller than $\unif(0,1)$ under any $\prob$}. 
\end{equation}
\citet[][Cor.~2.9]{Vovk:2005} show that the plausibility contour returned by the conformal prediction algorithm satisfies \eqref{eq:pi.1} and, from this, Type-1 prediction validity follows.  

Towards Type~2 validity, suppose that the function $\pi$ also satisfies
\begin{equation}
\label{eq:pi.2}
\sup_{\tilde y} \pi(\tilde y; y^n) = 1 \quad \text{for all $y^n$}. 
\end{equation}
This property holds quite generally for conformal prediction in continuous-data problems: if $\hat y_n$ is a point at which the minimum of $\tilde y \mapsto \Psi(y^n, \tilde y)$ is achieved, then $\pi(\hat y_n; y^n) = 1$ and \eqref{eq:pi.2} holds.  For discrete-data problems; see the discussion in \citet{cella.martin.discrete}. From \eqref{eq:pi.2}, one can readily define an upper prediction probability 
\begin{equation}
\label{eq:consonance}
\uPi_{y^n}(A) = \sup_{\tilde y \in A} \pi(\tilde y; y^n), \quad A \subseteq \YY. 
\end{equation}
This upper probability is a {\em consonant plausibility function} 
or, equivalently, a {\em possibility measure}, 
as discussed in Section~\ref{SS:ip}.  The plausibility contour $\pi$ fully determines the upper and lower prediction probabilities, through \eqref{eq:consonance}.  

From a practical point of view, adding consonance to conformal prediction creates no new computational challenges.  The standard use of Algorithm~\ref{algo:conformal}'s output is to extract the prediction region, $\plint_\alpha(y^n)$, which is just the collection of all $\tilde y$ such that $\pi(\tilde y; y^n)$ exceeds $\alpha$.  With the addition of consonance, we recommend two additional summaries.  First, at least in low-dimensional problems, a plot of $\tilde y \mapsto \pi(\tilde y; y^n)$ provides a nice visual assessment of the information available in the data $y^n$ regarding $Y_{n+1}$, similar to the Bayesian posterior predictive density function; see Section~\ref{S:examples}. 
Second, for any assertion $A \subset \YY$, the prediction upper probability at $A$ can be approximated as 
\[ \uPi_{y^n}(A) \approx \max_{\text{$\tilde y$ on the grid and in $A$}} \pi(\tilde y; y^n). \]
Consonance also induces several practically relevant properties.  First, consonance implies that the probabilistic predictor is {\em coherent} in the sense of \citet{walley1991}; see Proposition~7.14 in \citet{lower.previsions.book}.  Moreover, it turns out that consonance is crucial to establishing Type-2 validity.

\begin{thm}
\label{thm:strong}
The probabilistic predictor defined by the consonant plausibility function in \eqref{eq:consonance}, determined by a contour $\pi$ satisfying \eqref{eq:pi.1} and \eqref{eq:pi.2}, achieves strong Type-2 validity in Definition~\ref{def:type3}. 
\end{thm}

\begin{proof}
Combine \eqref{eq:pi.1} with Proposition~\ref{prop:type3}.
\end{proof}


To our knowledge, the strong Type-2 prediction validity property of conformal prediction is new.  By converting the output of the conformal prediction algorithm into a consonant plausibility function one achieves a stronger validity property beyond coverage probability of prediction regions.  The stronger validity property is valuable because it ensures that errors in prediction---{\em for all possible assertions about $Y_{n+1}$}---are uniformly controllably rare. 



In a recent paper, \citet{Vovk2018NonparametricPD}  
developed a nonparametric predictive probability distribution based on conformal prediction. Roughly speaking, by using a suitably monotone non-conformity measure, they obtain a similarly monotone conformal transducer, which is then interpreted as a predictive probability distribution for $Y_{n+1}$, given $y^n$.  The motivation behind this interpretation is that ``a conformal predictive distribution contains more information'' \citep[][p.~472]{Vovk2018NonparametricPD} than conformal prediction intervals.  It is true that a probability distribution can be used to calculate probabilities for arbitrary assertions $A$ about $Y_{n+1}$.  However, based on the discussion above and in Section~\ref{SS:type3}, the conformal predictive distribution is not strong Type-2 valid, hence, interpretation of those probabilities is difficult and their use may lead to erroneous predictions.  Indeed, while the conformal predictive distribution function evaluated at $Y_{n+1}$ is calibrated, the probability assigned to intervals $A$ that contain $Y_{n+1}$---i.e., the difference of the distribution function evaluated at the two endpoints---could be much smaller than this, leading the user to systematically ``reject'' such an assertion even though it is true.  Coupling the conformal prediction algorithm's output with consonance and the plausibility calculus as described in \citet{shafer1976mathematical}, however, leads to the strong Type-2 validity property, which ensures the (lower and upper) prediction probabilities are uniformly calibrated and, therefore, practically meaningful in the sense of Cournot.  


\section{Achieving Type-2 validity, II: a new class of nonparametric inferential models}
\label{S:im}

\subsection{Objectives and background}

The previous section showed that by interpreting the conformal prediction output as the contour function that defines a full consonant plausibility function, a Type-2 validity property can be achieved, one that goes beyond coverage probabilities of prediction regions.  This connection between conformal prediction and consonant plausibility functions suggests that insights about conformal prediction can be gained from the perspective of imprecise probabilities---in particular, random sets, possibility measures, etc.  In this section, we show that there is an alternative route to ``conformal + consonance'' as described above, through the IM framework, due to \citet{mainMartin, martinbook}, based on distributions of random sets.  This connection to IMs provides a new, imprecise probability-based characterization of conformal prediction.  

The IM approach has close connections to fiducial inference \citep{fisherfiducial, taraldsen.lindqvist.2013}, generalized fiducial inference \citep{MainHaning}, structural inference \citep{fraser1968structure}, Dempster--Shafer theory \citep{dempster.copss, dempster1967, dempster1968a, DEMPSTER2008365, shafer1976mathematical}, and others, including Bayesian inference \citep[e.g.,][Remark~4]{condmartin}, confidence structures \citep{balch2012}, and other calibrated beliefs frameworks \citep{denoeux.li.2018, denoeux2014}. Its construction proceeds as follows.  Step~1 is to {\em associate} the observable data and unknown quantity of interest with an unobservable auxiliary variable.  This association usually characterizes the distribution of data, given the unknowns, but see below.  Step~2 is to {\em predict} the unobserved value of the auxiliary variable with a suitable, user-defined random set.  Finally, Step~3 is to map this random set to the space where the quantity of interest resides and then {\em combine} it with the association at the observed data value.  This yields a data-dependent random set whose distribution is used to quantify uncertainty about the unknowns, in particular, yielding lower and upper probabilities for any relevant assertions.  Easy to arrange properties of the user-defined random set ensure that the inference is valid in a sense similar to that described in Section~\ref{SS:strong} above.  For the readers' convenience, we present some details in Appendix~\ref{S:impred.old} about IMs for prediction under a parametric statistical model.  

Like Bayes, fiducial, and other approaches to statistical inference, the IM framework is {\em model-based}.  This is a significant obstacle because, for prediction, the goal is to make as few model assumptions as possible.  Here we construct a valid prediction IM assuming only exchangeability, borrowing on some key insights developed in \citet{martin2015,MARTIN2018105}.  That is, that the association---see the A-step below---need not  characterize the data-generating process. Instead, an association can be built using suitable functions of the observable data.  Below, these ``functions of the observable data'' are closely related to the non-conformity scores from conformal prediction. 

Connections between our IM-based construction of imprecise probabilities for prediction can be made with other more classical approaches, including {\em nonparametric predictive inference} as presented in \citet{AUGUSTIN2004251}, \citet{CoolenBayes}, \citet{augustin.etal.bookchapter}, and the references therein.  Some relevant comments about these connections are given in Appendix~\ref{S:classical}.

\subsection{Construction}
\label{SS:construction}

\subsubsection{A-step}
\label{SSS:new.Astep}

According to \citet{martin2015,MARTIN2018105}, the association between the observable data, quantity of interest, and unobservable auxiliary variables can be generalized, in particular, it need not fully characterize the data-generating process.  That is, certain non-invertible transformations can be considered and, as long as those transformations preserve the exchangeability---our only model assumption---in the observable data, then we can construct a valid IM for predicting $Y_{n+1}$ based on $Y^n$.   

Recall the setup in Section~\ref{SS:conformal} above, wherein we make a transformation from $Y^{n+1}$ to $T^{n+1}$, where $T^{n+1}=(T_1,\ldots,T_{n+1})$, with
\[ T_i = \psi_i(Y^{n+1}) = \Psi(Y_{-i}^{n+1}, Y_i), \quad i \in \I_{n+1}, \]
and $\Psi$ a user-specified non-conformity measure.  
The crucial feature is that, by the symmetry of $\Psi$, the image $T^{n+1}$ inherits the exchangeability of $Y^{n+1}$.  Where necessary in what follows, we will highlight $T_i$'s dependence on the data $Y^{n+1}$ by writing it as $T_i(Y^{n+1})$.

The joint distribution of $T^{n+1}$ is complicated.  A complete characterization requires a common marginal distribution function $G$ and, say, a copula that induces the exchangeable dependence structure.  In particular, marginally we have 
\begin{equation}
\label{eq:generalized.marginal}
 T_i = G^{-1}(U_i), \quad i \in \I_{n+1},  
\end{equation}
where the $U_i$'s are iid $\unif(0,1)$ and $G$ is an infinite-dimensional nuisance parameter.  Our goals are, first, to avoid directly dealing with $G$ (and the associated exchangeability-inducing copula) and, second, since we aim to predict the single data point $Y_{n+1}$, to reduce the dimension so that specifying a random set for the complicated and relatively high-dimensional $U^{n+1}$ is unnecessary.  

For simplicity, suppose for the moment that the $T_i$'s are continuous.  This would hold, e.g., if the $Y_i$'s are continuous and the non-conformity measure $\Psi$ is non-constant on sets of $Y^{n+1}$ with positive $\prob$-probability. The key observation is that, since $G$ is strictly increasing, the ranks of the $T_i$'s are well-defined (i.e., no ties) and the same as those of the $U_i$'s.  Moreover, exchangeability implies that the latter ranks are marginally discrete uniform on $\I_{n+1}$, denoted by $\unif(\I_{n+1})$.  This creates an opportunity to both eliminate the dependence on $G$ (and the copula) and reduce the dimension.  The quantity $T_{n+1}$ gives the to-be-predicted value $Y_{n+1}$ a special status and, as we just pointed out, its rank is uniformly distributed.  This suggests a dimension-reduced (generalized) association,
\begin{equation}
\label{eq:conformal.assoc}
r(T_{n+1}) = V, \quad V \sim \unif(\I_{n+1}), 
\end{equation}
where $r(\cdot)$ is the ranking operator, that depends implicitly on $T^{n+1}$ and, hence, on $Y^{n+1}$.  Here we assign rank 1 to the smallest value, rank 2 to the second smallest, and so on, because small values of the non-conformity measure are ``better'' in a certain sense.  That the values are ranked in ascending versus descending order will be important in the P-step below.  For now, we have completed the A-step: \eqref{eq:conformal.assoc} defines generalized association---in the sense of \citet{martin2015,MARTIN2018105}---that links the data $Y^n$ and the to-be-predicted value $Y_{n+1}$ to an auxiliary variable $V$ with known distribution.  

On the other hand, if the $T_i$'s are not continuous, then a different argument is required.  The problem is that, in this case, the marginal distribution function $G$ is not strictly increasing, which means the $T_i$'s can have ties, hence the ranks do not correspond to the ranks of the $U_i$'s in \eqref{eq:generalized.marginal}.  One can still make the reduction to the lower-dimensional feature $r(T_{n+1})$ as above, but its distribution is no longer completely known---some features of $G$ remain.  Fortunately, there is a simple remedy for this based on the following observation: when ties are possible, $r(T_{n+1})$ is stochastically no larger than when ties are not possible.  Therefore, following the logic in \citet[][Sec.~5]{marginalmartin}, the aforementioned association \eqref{eq:conformal.assoc} can still be used to construct a valid IM for predicting $Y_{n+1}$ as described below.  We sacrifice a bit of efficiency---compared to the exact/continuous case above---in order to eliminate the dependence on the nuisance parameter $G$.



\subsubsection{P-step}
\label{sss:p-step}

Towards valid prediction of $Y_{n+1}$, here our intermediate goal is to specify a random set targeting the unobserved realization of the auxiliary variable $V$ introduced above.  The existing literature on this has focused exclusively on cases where the auxiliary variable being targeted has a continuous distribution.  We can still follow the developments in \citet{martinbook}, but there are some differences in our present case of discrete $V$.  
Let $\S \sim \prob_\S$ denote a random set taking values on the power set of $\I_{n+1}$; this is a finite space, so the distribution $\prob_\S$ of the random set can be characterized simply by a mass function.  The only feature of $\prob_\S$ needed here is the contour function of $\S$, given by 
\begin{equation}
\label{eq:contourfunction}
\gamma_\S(v) := \prob_\S(\S \ni v), \quad v \in \I_{n+1}. 
\end{equation}
It will be important in what follows that the distribution of $\gamma_\S(V)$, as a function of $V \sim \unif(\I_{n+1})$ is as close to uniform as possible; see \eqref{eq:S.valid} and Lemma~\ref{lem:S.valid}.  This boils down to making a suitable choice of $\S$ targeting the unobserved value of $V$.

Towards this, define the random set 
\begin{equation}
\label{eq:S}
\S = \{1,2,\ldots,\tilde V\}, \quad \tilde V \sim \unif(\I_{n+1}),
\end{equation}
which is the push-forward of the $\unif(\I_{n+1})$ distribution through the set-valued mapping $v \mapsto \I_v$.  This random set makes intuitive sense because, since the ranking operator is relative to ascending order, so that ``rank equals 1'' corresponds to a prediction consistent with the observed data, the random set $\S$ should include value 1.  Furthermore, $\S$ in \eqref{eq:S} is {\em valid} for predicting an auxiliary variable $V \sim \unif(\I_{n+1})$ in the sense that 
\begin{equation}
\label{eq:S.valid}
\prob_V\{\gamma_\S(V) \leq \alpha\} \leq \alpha , \quad \alpha \in [0,1].
\end{equation}
In words, $\S$ is valid if $\gamma_\S(V)$, as a function of $V \sim \unif(\I_{n+1})$, is stochastically no smaller than $\unif(0,1)$. 





\begin{lem}
\label{lem:S.valid}
The random set $\S$ defined in \eqref{eq:S} is valid in the sense of \eqref{eq:S.valid}.  
\end{lem}

\begin{proof}
By direct calculation, we have 
\[ \gamma_\S(v) = \prob_\S(\S \ni v) = \prob_{\tilde V}(\tilde V \geq v) = 1- \tfrac{v-1}{n+1}, \quad v \in \I_{n+1}. \]
Since $\gamma_\S(V)$ is a linear function of $V \sim \unif(\I_{n+1})$, its exact distribution is $\unif\{(n+1)^{-1}\I_{n+1}\}$, which is stochastically no smaller than $\unif(0,1)$, proving the claim.
\end{proof}

The proof of Lemma~\ref{lem:S.valid} shows that $\gamma_\S(V)$ is exactly $\unif\{(n+1)^{-1}\I_{n+1}\}$ distributed, not just stochastically no smaller.  This is an important characteristic because it implies that the set $\S$ in \eqref{eq:S} is efficient in the following sense.  First note that the constant random set $\S \equiv \I_{n+1}$ is valid but, because it is the full $V$-space, it cannot provide any valuable information.  So, at least intuitively, we seek the ``smallest'' random set that is valid.  If we measure the ``size'' of a random set $\S$ via its covering function, $\gamma_\S$, and if $\S'$ is another random set that is smaller than $\S$ in the sense that 
\[ \gamma_{\S'}(v) < \gamma_\S(v), \quad \text{for some $v \in \I_{n+1}$},\]
then it is easy to check that $\S'$ cannot satisfy \eqref{eq:S.valid}.  Therefore, we complete the P-step by recommending use of the random set $\S$ in \eqref{eq:S} targeting the unobserved auxiliary variable $V$ in the A-step from Section~\ref{SSS:new.Astep} since it is both valid and {\em efficient}.

\subsubsection{C-step}
\label{ss:cstep}
Next we combine the results of the A- and P-steps to construct a new random set on $\YY$ whose distribution determines a probabilistic predictor for quantifying uncertainty about $Y_{n+1}$.  If we express the result of the A-step with the $v$-indexed collection of sets 
\begin{equation}
\label{eq:v-indexed}
\YY_{y^n}(v) = \{y_{n+1}: r(T_{n+1}(y^{n+1})) = v\}, \quad v \in \I_{n+1}, 
\end{equation}
and then combine this with the random set $\S$ in \eqref{eq:S} from the P-step above, then the C-step yields the new random set on $\YY$, given by 
\[ \YY_{y^n}(\S) = \{y_{n+1}: r(T_{n+1}(y^{n+1})) \leq \tilde V\}, \quad \tilde V \sim \unif(\I_{n+1}). \]
Following the general random set theory \citep[e.g.,][]{nguyen2006introduction, molchanov2005}, define the contour---or covering probability---function of the random set $\YY_{y^n}(\S)$, which is determined by the distribution of $\S$, as 
\begin{align}
\label{eq:plausonedim}
\pi(y_{n+1}; y^n) & = \prob_\S\{\YY_{y^n}(\S) \ni y_{n+1}\} \nonumber \\
& = \prob_{\tilde{V}}\{\tilde{V} \geq r(T_{n+1}(y^{n+1}))\} \nonumber \\
& = \frac{1}{n+1} \sum_{i=1}^{n+1} 1_{T_i(y^{n+1}) \geq T_{n+1}(y^{n+1})}. 
\end{align}
We immediately recognize the right-hand side above as conformal prediction's plausibility contour as in Algorithm~\ref{algo:conformal}.  This establishes the connection between conformal prediction and the new class of nonparametric IMs.  Moreover, since the random set $\S$ is nested, its data-dependent image $\YY_{y^n}(\S)$ on the $\YY$ space is nested too, and the distribution of a nested random set corresponds to a consonant plausibility function.  Technically, 
\[ \uPi_{y^n}(A) := \prob_\S\{\YY_{y^n}(\S) \cap A \neq \varnothing\} = \sup_{\tilde y \in A} \pi(\tilde y; y^n), \quad A \subseteq \YY. \]
Consequently, our recommendation in Section~\ref{S:consonance} to construct a consonant plausibility function from the conformal prediction algorithm's output was not an ad hoc adjustment to achieve strong validity.  Rather, as the above derivation suggests, the connection between conformal prediction, consonant plausiblity functions, and IMs is fundamental.

\subsection{Strong Type-2 validity}
\label{SS:validity}

Of course, since the IM's consonant plausibility function exactly matches that from conformal prediction, the same strong Type-2 prediction validity property established in Section~\ref{S:consonance} must hold here too.  Here we give a direct proof---following the general theory of IMs, as in \citet{martinbook}---to showcase the important role played by the nested random set $\S$ in the IM construction.  

\begin{thm}
\label{thm:valid}
The probabilistic predictor defined by the consonant plausibility function derived above, using a random set $\S$ satisfying \eqref{eq:S.valid}, is strongly Type-2 valid for prediction in the sense of Definition~\ref{def:type3}.   
\end{thm}

\begin{proof}
We proceed by checking condition \eqref{eq:pi.1} for the plausibility contour $\pi(Y_{n+1}; Y^n)$ defined by the left-hand side of \eqref{eq:plausonedim}.  Towards this, for a realization of the random set $\S$ in \eqref{eq:S}, define the random variable $\overline{\S} = \max\S$, the set's upper bound.
By definition of the plausibility contour, we have
\begin{align*}
\prob\{\pi(Y_{n+1}; Y^n)\leq \alpha\} &= \prob\bigl[\prob_{\S}\{ \YY_{Y^n}(\S) \ni Y_{n+1} \}\leq \alpha\bigr] \\
&= \prob\bigl[\prob_{\S}\{ r(T_{n+1}(Y^{n+1})) \leq \overline{\S} \}\leq \alpha \bigr] \\
&= \prob_{V}\bigl\{\prob_{\S}(V \leq \overline{\S})\leq \alpha \bigr\} \\
&= \prob_{V}\{\gamma_{\S}(V)\leq \alpha\}.
\end{align*}  
Then the claim follows from Proposition~\ref{prop:type3}. 
\end{proof}



In Section~\ref{sss:p-step} we argued that the random set $\S$ in \eqref{eq:S} is efficient in the sense that smaller random sets would not be valid.  However, $\S$ itself is also inefficient in the sense that the corresponding prediction plausibility regions are conservative: their coverage probability generally exceeds the nominal level.  To improve the efficiency, we need to take a smaller random set but, unfortunately, the discreteness of the problem limits our flexibility.  There is no random set that is both valid and ``strictly smaller'' than $\S$, but we can reduce the size by introducing some extra {\em randomization}, as described below.  

Consider a family of random sets $\S^w$, indexed by $w \in [0,1]$, given by:
\begin{align}\label{Sw}
\S^w = \begin{cases} \{1,\ldots,\tilde V\} & \text{if $\xi=0$} \\ \{1,\ldots,\tilde V - 1\} & \text{if $\xi=1$ and $\tilde V > 1$} \\ \varnothing & \text{if $\xi=1$ and $\tilde V=1$}, \end{cases} \qquad (\tilde V,\xi) \sim \unif(\I_{n+1}) \times \ber(w),
\end{align}
where $\ber(w)$ denotes a Bernoulli distribution with success probability $w$. In words, $\S^w$ flips a $w$-coin to decide between $\S=\{1,\ldots,\tilde V\}$ in \eqref{eq:S} and the smaller $\{1,\ldots,\tilde V-1\}$, both using the same $\tilde V$.  With fixed $w$, it is easy to check that 
\[ \gamma_{\S^w}(v) = \prob_{\S^w}(\S^w \ni v) = 1-(n+1)^{-1}(v-w). \]
A relatively simple calculation reveals that, if $w$ is treated as a uniformly distributed random variable $W$, i.e., if $(V,W) \sim \unif(\I_{n+1}) \times \unif(0,1)$, then 
\begin{equation}
\label{eq:vw}
(V,W) \mapsto \gamma_{\S^W}(V) \sim \unif(0,1). 
\end{equation}
So if we couple the data sequence $Y_1,Y_2,\ldots$ with a corresponding sequence $W_1,W_2,\ldots$ of independent $\unif(0,1)$ random variables that we generate, then we can proceed by constructing an IM for predicting $Y_{n+1}$ by using the random set $\S^{w_{n+1}}$ corresponding to the observed value $w_{n+1}$ of $W_{n+1}$.  This produces a $w_{n+1}$-dependent plausibility contour 
\begin{align}\label{eq:randplaus}
\pi^{w_{n+1}}(y_{n+1}; y^n) & = \prob_{\S^{w_{n+1}}}\{\YY_{y^n}(\S^{w_{n+1}}) \ni y_{n+1}\} \nonumber \\
& = \frac{1}{n+1} \sum_{i=1}^{n+1} 1_{T_i(y^{n+1}) > T_{n+1}(y^{n+1})}  \nonumber \\
& \qquad + \frac{w_{n+1}}{n+1} \sum_{i=1}^{n+1} 1_{T_i(y^{n+1}) = T_{n+1}(y^{n+1})}. 
\end{align}
It follows immediately from \eqref{eq:vw} that
\begin{equation}
\label{eq:randval}
\prob\{\pi^{W_{n+1}}(Y_{n+1}; Y^n) \leq \alpha\}  = \alpha \quad \text{for all $n$ and all $\alpha \in [0,1]$}, 
\end{equation}
with the probability taken over $(Y^{n+1}, W_{n+1})$ now, which implies that exact coverage probability for the corresponding {\em randomized} prediction plausibility region.  Strong Type-2 validity holds as well.  One will also immediately recognize $\pi^{w_{n+1}}(y_{n+1}; y^n)$ as conformal prediction's ``smoothed p-value'' in, e.g., Equation~(2.20) of \citet{Vovk:2005}.  

The downside to this exact prediction framework is two-fold: first, as a result of the $W$-dependence, two data analysts could produce different prediction intervals with the same $Y_1,Y_2,\ldots$ data sequence; second, it requires inappropriate use of a random set taking value $\varnothing$ with non-zero probability.  To us, both of these are legitimate concerns, so we include this exact prediction validity result here only for theoretical interest.

\subsection{Optimality}
\label{ss:Optimality}

\citet{martin2017mathematical} established a connection between valid IMs and confidence regions.  He shows that, given a confidence region for some feature $\phi = \phi(\theta)$ of the full parameter, there exists a valid inferential model for $\theta$ whose corresponding
marginal plausibility region for $\phi$ is at least as efficient as the given confidence region.  Given the flexibility of the nonparametric IMs for prediction presented above, one may wonder if a similar result is possible for prediction intervals based solely on an exchangeability assumption. 

\citet{Vovk:2005} and \citet{Shafer2007ATO} present an optimality result in the context of conformal prediction. Roughly, given any valid and nested prediction regions that are invariant with respect to the
permutations of $y^n$, there exists a non-conformity measure such that the corresponding conformal prediction region derived from it is at least as efficient. More precisely, if $\mathscr{C}_\alpha(y^n)$ represents a valid and nested $100(1-\alpha)$\% prediction region for each $\alpha \in [0,1]$, then there exists a nonconformity measure---determined by a mapping $\Psi$---whose corresponding $100(1-\alpha)$\% conformal prediction region $\plint_\alpha(y^n) = \plint_\alpha(y^n; \Psi)$ satisfies 
\[ \plint_\alpha(y^n) \subseteq \mathscr{C}_\alpha(y^n), \quad \text{for all $y^n$ and all $\alpha \in [0,1]$}. \]
The function $\Psi$ can even be identified explicitly, i.e., 
\[ \Psi(\tilde y^n; \tilde y) = \sup\{\alpha \in [0,1]: \mathscr{C}_\alpha(\tilde y^n) \not\ni \tilde y\}. \]
We showed above that every conformal prediction region is a valid IM's plausibility region which, together with Shafer and Vovk's argument, proves our present claim.  That is, given any valid, nested, and permutation invariant family of prediction regions, there exists a valid IM whose prediction plausibility regions are at least as efficient.

\section{Examples}
\label{S:examples}

This section presents several numerical examples intended to highlight two things.  First, as suggested by the ``conformal + consonance'' characterization in Section~\ref{S:consonance}, the importance of the conformal p-value/tranducer as output.  Indeed, as described there, this function can be plotted for a visual assessment of which values of the future observable are most plausible based on the observed data.  Using the plausibility calculus, one can get a rough estimate of the plausibility assigned to any assertion just by looking at this plot, at least in one-dimensional problems.  Second, to highlight the new IM-based characterization in Section~\ref{S:im}, we frame our presentation here using that terminology.

\subsection{One-dimensional prediction}
\label{ss:one}

Consider a random sample of size $n=50$ of some scalar random variable $Y$; a histogram of these data is shown in Figure~\ref{fig:histogram}(a). The goal is to construct a valid nonparametric IM for prediction of $Y_{51}$.

Figure \ref{fig:histogram}(b) shows, in red, the plausibility contour in \eqref{eq:plausonedim} with
\begin{equation}\label{eq:median}
\psi_i(y^{n+1}) = |\text{median}(y_{-i}^{n+1}) - y_i|, \quad i \in \I_{n+1}, \quad n=50.  
\end{equation}
This plausbility contour has its mode at the sample median and is symmetric.  For comparison, we also consider a different non-conformity measure, namely,
\[
\psi_i(y^{n+1}) = y_i, \quad i \in \I_{n+1}, \quad n=50.  
\]
Note that the structure of this function $\psi_i$ is inconsistent with the choice of the random set in \eqref{eq:S} suggested in Section~\ref{S:im}.  But it turns out that this choice of non-conformity measure has close connections to classical nonparametric predictive inference and the textbook order statistics-based prediction intervals \citep[e.g.,][]{wilks1941}, so we give a complete description of the IM formulation in this case in Appendix~\ref{S:classical}. The resulting plausibility contour, as given in \eqref{eq:twosided}, is shown, in blue, in Figure~\ref{fig:histogram}(b). Note that, by thresholding it at 0.05 we obtain the 95\% classic prediction interval $[y_{(1)},y_{(50)}]$. Moreover, the plausibility contour based on \eqref{eq:plausonedim} is narrower than that based on \eqref{eq:twosided}, which indicates that the former is more efficient than the latter. For further comparison, the plausibility contour in \eqref{eq:plausonedim} with non-conformity measure as in \eqref{eq:mean} is also shown in Figure~\ref{fig:histogram}(b), in black. As expected, there is not much difference between the plausibility contours based on the mean- and median-based non-conformity measures. 


\begin{figure}[t]
\begin{center}
\subfigure[Histogram of $y^{50}$]{\scalebox{0.5}{\includegraphics{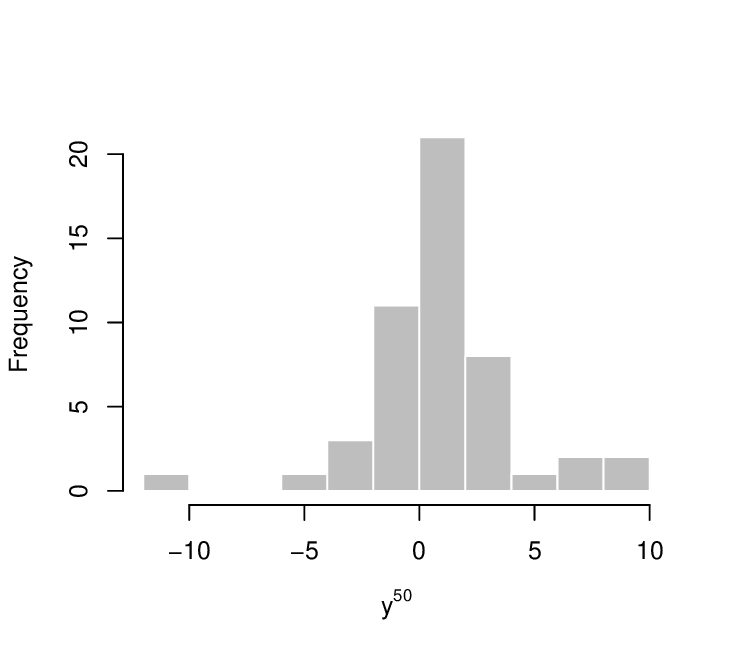}}}
\subfigure[Plausibility contour for $Y_{51}$]{\scalebox{0.5}{\includegraphics{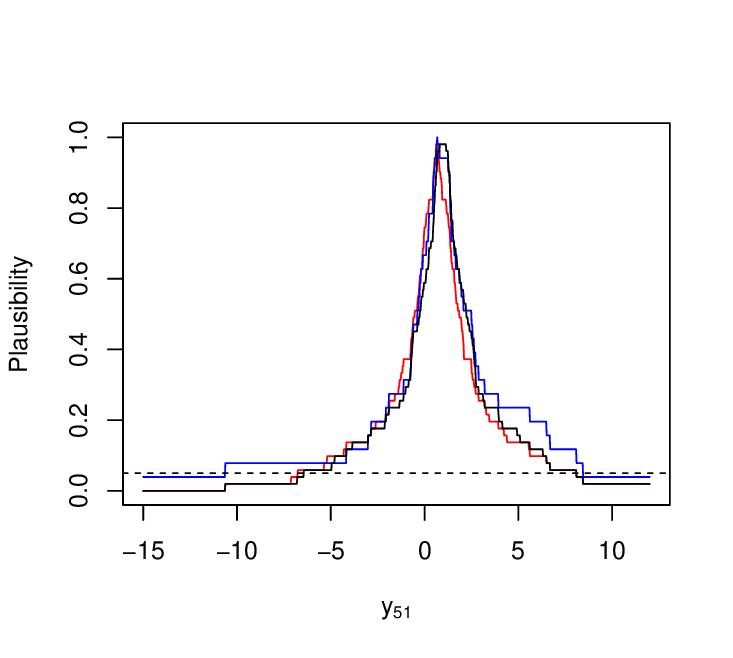}}}
\end{center}
\caption{Panel (a): Histogram of the observed data $y^{50}$. Panel (b): Plausibility contours in Equations~\eqref{eq:twosided}, \eqref{eq:plausonedim} and \eqref{eq:mean}, colored blue, red and black, respectively.}
\label{fig:histogram}
\end{figure}

For further illustration, we reconsider the simulation experiment described in Section~\ref{SS:weak}. That is, there are three different true data-generating distributions, namely, the normal distribution with mean zero and variance 0.5, the standard Cauchy distribution, and the standard skewed normal distribution \citep{azzalini} with skewness parameter equal to 1, with three different sample sizes: $n \in \{20, 30, 40\}$. For each of these combinations of data-generating distribution and sample size, we evaluate the coverage probability of 90\% nonparametric IM prediction intervals with non-conformity measure as in \eqref{eq:median}, based on 5000 Monte Carlo samples. The results are shown in Table~\ref{tab:coverage2}. As expected, the nonparametric IM achieves validity in the sense of \eqref{eq:weak} regardless of the data distribution and sample size.

\begin{table}[t]
\centering
\begin{tabular}{cc c c}
\hline
$n$ & Normal & Cauchy & Skew Normal \\
\hline
20 & 0.905 & 0.898 & 0.904  \\
30 & 0.897 & 0.902 & 0.901  \\
40 & 0.898 & 0.899 & 0.904\\
\hline
\end{tabular}
\caption{Estimated coverage probabilities of 90\% prediction intervals for $Y_{n+1}$, based on the nonparametric IM with non-conformity measure as in \eqref{eq:median}, with three different data-generating distributions.}
\label{tab:coverage2}
\end{table}

\subsection{Two-dimensional prediction}
\label{ss:depth}

An advantage of conformal prediction and, consequently, the nonparametric IM developed in Section~\ref{S:im}, is its flexibility to deal with multivariate responses. Algorithm~\ref{algo:conformal} can be followed to compute $\pi(y_{n+1};y^n)$, independent of the dimension. This is not the case for other nonparametric prediction methods. For example, the classical prediction interval methods depend on order statistics, and, therefore, cannot be generalized for multivariate responses since there is no natural ordering in this context \citep{Oja2013}.  

Computation of $\pi(y_{n+1};y^n)$ requires the specification of a non-conformity measure $\Psi$. For multivariate responses, there are a number of options, for example, \citet{leipredictionset} construct conformal prediction regions using a multivariate kernel density estimator.  As an alternative, here we use {\em data depth} \citep[e.g.,][]{liu1999} to construct a non-conformity measure and, in turn, a prediction plausibility contour as in \eqref{eq:plausonedim}.  

Roughly speaking, the concept of data depth amounts to assigning an appropriate ordering to the multivariate data, one with respect to a distance that measures how far away a point is from the center of a data cloud.  This ordering is determined by a depth function and, among the variety of depth functions that appear in the literature, one of the most widely used is the so-called {\em half-space depth} proposed by \citet{Tukey1975MathematicsAT}.  The half-space depth, $H(y \mid y^n)$, of a point $y$ relative to a data set $y^n$ is determined as the smallest fraction of data points contained in a closed half-space with boundary
through $y$. It ranges from 0---for the points that lie beyond the convex hull of the data---to its maximum value $\frac12$ attained at the Tukey median \citep{DYCKERHOFF201619}.  Computation of the half-space depth is discussed in \citet{CUESTAALBERTOS20084979} and can be carried out using the {\tt depth.halfspace} function in {\em ddalpha} R package \citep{depthR}.  Therefore, for multivariate prediction problems, this suggests the following non-conformity measure 
\begin{equation}
\label{eq:depth}
\psi_i(y^{n+1}) = \tfrac12 - H(y_i \mid y_{-i}^{n+1}), \quad i \in \I_{n+1},  
\end{equation}
where $H$ is the half-space depth function defined above.  

As an illustration, consider the bivariate data $y^{60} = (y_1,\ldots,y_{60})$, where each $y_i$ consists of the weight and engine displacement of one of 60 cars; these data are available in the R library {\em rpart} and are plotted in Figure~\ref{fig:biplaus}(a). The goal is to predict $Y_{61}$, the weight and engine displacement of the 61st car.  Figures~\ref{fig:biplaus}(b), \ref{fig:biplaus}(c) and \ref{fig:biplaus}(d) show, in different angles, the plausibility contour in \eqref{eq:plausonedim} based on the nonparametric IM with non-conformity measure in \eqref{eq:depth}. The shaded area in Figure \ref{fig:biplaus}(a) represents the 80\% prediction plausibility region for $Y_{61}$.

\begin{figure}[t]
\begin{center}
\subfigure[]{\scalebox{0.5}{\includegraphics{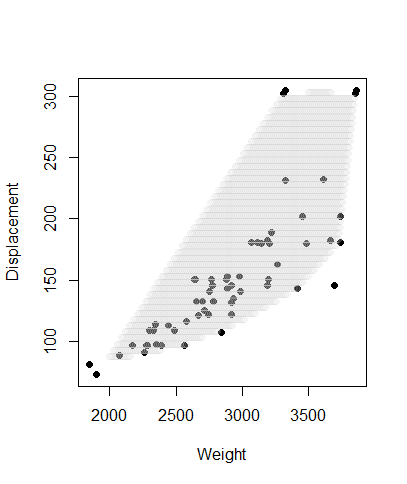}}}
\subfigure[]{\scalebox{0.6}{\includegraphics{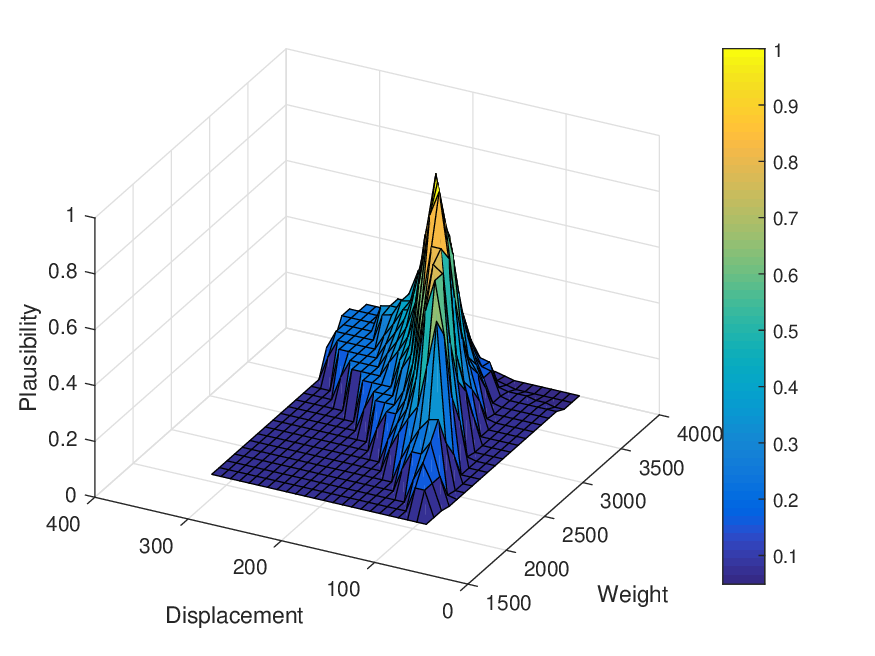}}}
\subfigure[]{\scalebox{0.5}{\includegraphics{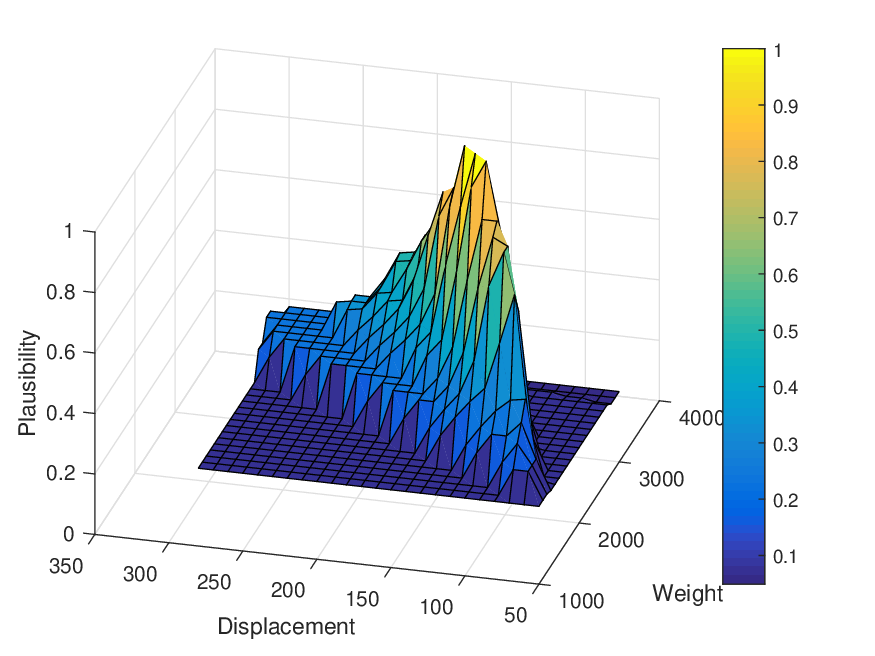}}}
\subfigure[]{\scalebox{0.5}{\includegraphics{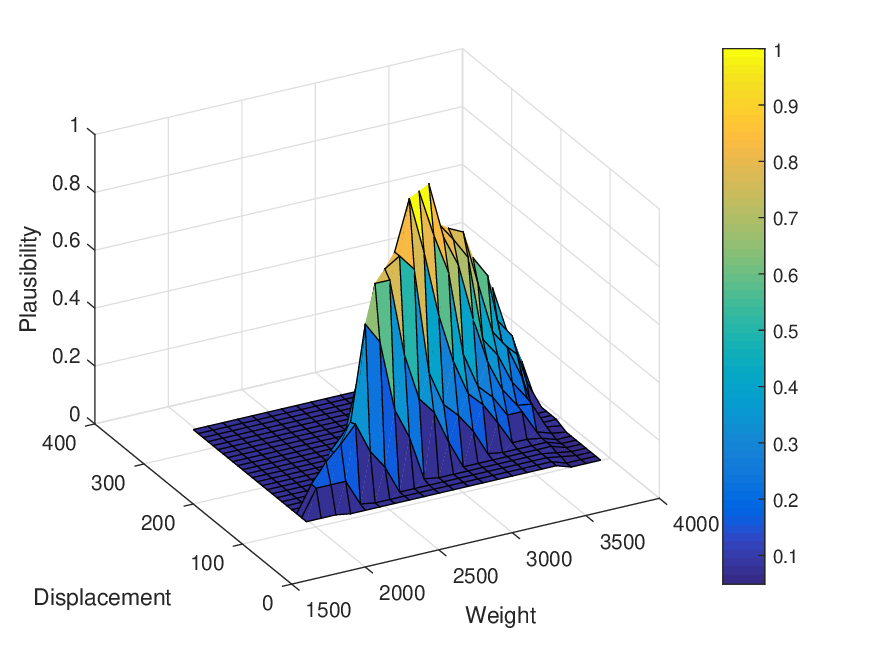}}}
\end{center}
\caption{Panel (a): Scatter plot of $y^{60}$ and the 80\% prediction plausibility region shaded gray.  Panels (b), (c), and (d): Plausibility contour in Equation~\eqref{eq:plausonedim} from different angles.}
\label{fig:biplaus}
\end{figure}

\section{Conclusion}
\label{S:discuss}

In this paper, we considered the fundamental problem of valid probabilistic prediction.  The literature on this topic tends to focus on solely on achieving prediction coverage probability guarantees, what we refer to as {\em Type-1 prediction validity} in Definition~\ref{def:weak}.  Here, however, we claim that, if there is reason to consider a ``predictive distribution'' as opposed to just prediction regions, then a different notion of validity---what we refer to as {\em Type-2 prediction validity} in Definition~\ref{def:strong}---is required.  We go on to show Type-1 and 2 prediction validity can be achieved by a class of imprecise probabilistic predictors, namely, the consonant plausibility functions whose contour corresponds to a conformal transducer.  To our knowledge, this connection between consonant plausibility functions and conformal prediction---two of Shafer's many significant contributions---through this notion of Type-2 validity is new.  In addition, we provide a novel characterization of conformal prediction through the IM framework, which makes the connection to consonant plausibility functions (via nested random sets) more direct and natural.  This connection also sheds light on the power, flexibility, and fundamental nature of the IM framework for statistical inference and prediction.  


Even before introducing our notion of Type-2 validity in Section~\ref{SS:strong}, we had doubts in Section~\ref{SS:weak} that a precise probabilistic predictor could achieve Type-1 validity.  This claim was based on a parallel result concerning consonance, coverage probability, and dominance in the imprecise probability literature.  Consider the simpler case where there is a single uncertain outcome, as opposed to a sequence of prediction problems where the goal is to quantify uncertainty about $Y_{n+1}$, given $Y^n=y^n$.  As summarized in Propostion~4.1 of \citet{destercke.dubois.2014} and the relevant discussion, given a collection $\{C_\alpha: \alpha \in [0,1]\}$ of nested sets, if we form the credal set consisting of all probability measures $P$ such that $P(C_\alpha) \geq 1-\alpha$, then its upper envelope is a consonant plausibility function and the sets $C_\alpha$ are the $\alpha$-cuts of the corresponding plausibility contour function; see, also, \citet{cuoso.etal.2001} and \citet{dubois.etal.2004}.  The setting just described is different from that of the present paper, but there are some obvious parallels that suggest a similar conclusion ought to hold here too.  Indeed, the sets $C_\alpha$ are like our prediction regions and ``$P(C_\alpha) \geq 1-\alpha$'' is like our desired prediction coverage probability in \eqref{eq:weak}.  We want this coverage property to hold for as many distributions as possible, hence we consider the corresponding credal set, and the result just described says that its upper envelope is a consonant plausibility function and, moreover, the sets $C_\alpha$ are the $\alpha$-cuts of the corresponding plausibility contour.  If this connection could be made rigorous, then it would follow that our ``conformal + consonance'' is the only way validity can be achieved.  Surely the conditional nature of the present prediction problem creates some unique challenges, so this makes for an interesting open question. 


The IM construction presented here is unique in the sense that it is not based on an association determined by a statistical model.  Instead, the {\em generalized association} is made through some very basic distributional properties resulting from the apparently innocuous assumption of exchangeability.  This begs the question if this ``model-free'' form of generalized association can be used in the context of statistical inference as opposed to prediction.  A first step would be the development of IM-based solutions to the classical nonparametric problems, e.g., inference on quantiles without distributional assumptions, moving on to more modern/complex problems.  Closely related to inference is that of assessing the quality or appropriateness of a posited statistical model, and here conformal-type methods have already proved to be useful \citep[e.g.,][]{wassermanregressionprediction, wassermanCommentSci}. 
That is, since a model is judged to be ``good'' only if it predicts the observed data reasonably well, methods that provide high-quality prediction can be leveraged to provide equally high-quality model assessment.  The work in the present paper provides a starting point for the construction of IM-based procedures for valid model assessment.  More generally, while IMs had previously focused on classical problems with a fully specified statistical model, 
the work presented in this paper reveals an opportunity for IMs to play a role in solving modern problems where valid inference/prediction is required with as few model assumptions as possible.

\section*{Acknowledgments}

The authors thank Vladimir Vovk, Jon Williams, and two anonymous reviewers, for their helpful feedback on an earlier version of this manuscript.  This work is partially supported by the U.S.~National Science Foundation, DMS--1811802.

\appendix

\section{Additional technical details}

\ifthenelse{1=1}{}{
\subsection{Support for the claim in Section~\ref{SS:strong}}
\label{S:conjecture}

Here we develop some more intuition about Type-2 validity and what it takes to satisfy \eqref{eq:strong} for the case where the $Y_i$'s take values in a finite set $\YY$.  

For a fixed $n$, if $|\YY|$ is the (finite) cardinality of the sample space $\YY$, then $R = |\YY|^n$ is the total number of possible configurations of the $n$-vector $y^n$.  Given an assertion $A$, we rank the different configurations of $y^n$ relative to the magnitude of $\uPi_{y^n}(A)$.  Write $y^n[r;A]$ for the configuration $y^n$ that is of rank $r$, relative to $A$, i.e., 
\[ \uPi_{y^n[r;A]}(A) > \uPi_{y^n[r+1;A]}(A), \quad \text{for all $r=1,\ldots,R$}. \]
For simplicity, we will assume that there are no ties.  The derivation is tedious, hence omitted, but it can be shown that the Type-2 validity condition \eqref{eq:strong} is equivalent to
\begin{equation}
\label{eq:joint}
\uPi_{y^n[r;A]}(A) \geq \sum_{k=r}^R \prob(Y_{n+1} \in A \, , \, Y^n = y^n[k;A]), \quad \text{for all $(r,n,A,\prob)$}.
\end{equation}
The ``for all $\alpha$'' part of the condition has been replaced by ``for all $r$'' here because there are only finitely many values.  Of course, the right-hand side of \eqref{eq:joint} can be rewritten in terms of conditional probabilities, leading to an equivalent form of Type-2 validity:
\[ \uPi_{y^n[r;A]}(A) \geq \sum_{k=r}^R \prob(Y_{n+1} \in A \mid Y^n = y^n[k;A]) \, \prob(Y^n = y^n[k;A]), \quad \text{all $(r,n,A,\prob)$}. \]
Notice that, if $r=1$ (top rank), then the sum in \eqref{eq:joint} is over all $y^n$, so the right-hand side reduces to $\prob(Y_{n+1} \in A)$.  Similarly, if $r=R$ (bottom rank), then \eqref{eq:joint} yields 
\[ \uPi_{y^n[R;A]}(A) \geq \prob(Y_{n+1} \in A \, , \, Y^n = y^n[R;A]), \quad \text{for all $(A,n,\prob)$}. \]
These more specific versions of \eqref{eq:strong} have inequalities that further support the ``dominance'' claim mentioned above, which suggest the probabilistic predictor cannot be a probability measure.  Unfortunately, the dependence on $A$-dependent ranks, etc., makes these inequalities difficult to interpret. To push this argument further, we will {\em assume} that $\uPi_{y^n} \equiv \Pi_{y^n}$ is a probability measure, which will allow us to simplify the above expressions and gain some further insights.  

In particular, if the probabilistic predictor is a probability measure $\Pi_{y^n}$, then we have additivity, $\Pi_{y^n}(A^c) = 1 - \Pi_{y^n}(A)$, for all $A$ and all $y^n$. This duality and the definition of the ranks above leads to the following key observation:
\[ \text{if $\Pi_{y^n}$ is a probability, then $y^n[r;A] = y^n[R-r+1; A^c]$ for all $r$ and all $A$}. \]
Since all the above inequalities have to hold for all assertions, we can take the bounds for both $A$ and $A^c$, and use the this key observation to pool the inequalities together.  Indeed, by considering the two extreme cases---$r=1$ and $r=R$---it can be shown that 
\begin{align*}
\prob(Y_{n+1} \in A, Y^n = y^n[R;A]) & \leq \min_{y^n} \Pi_{y^n}(A) \leq \prob(Y_{n+1} \in A) \\
\prob(Y_{n+1} \in A) & \leq \max_{y^n} \Pi_{y^n}(A) \leq \prob(Y_{n+1} \in A) + \prob(Y^n \neq y^n[1;A]), 
\end{align*}
again, for all $(A,n,\prob)$.  By considering intermediate values of $r$, we get further constraints on the values of the probabilistic predictor.  The critical point, however, is that these constraints depend on the true distribution $\prob$.  The probabilistic predictor generally does not have knowledge of the true distribution---if it did, then the right choice is for $\Pi_{y^n}$ to be the true conditional distribution, which satisfies all of the above constraints.  But without knowledge of $\prob$, then it would be impossible to arrange for the above expressions to hold for all $\prob$.  For example, if the data are coin flips, then the second inequality in the above display requires that we {\em assign} the probabilistic prediction of Heads to have $\Pi_{y^n}(\text{Heads})$ greater than the true marginal probability of Heads for at least one configuration $y^n$.  Since an assignment that ensures this for all $\prob$ would require knowledge of the true distribution, we have reason to doubt that there are probability distributions that can achieve Type-2 validity.  
}


\subsection{IMs for prediction in parametric models}
\label{S:impred.old}

Suppose that $(Y^n,Y_{n+1})$ are iid from $\prob_{Y|\theta}$. \citet{Martin2016PriorFreePP} showed that prediction can effectively be treated as a marginal inference problem, one where $\theta$ itself is a nuisance parameter.  Their arguments involve the dimension-reduction strategies in \citet{condmartin, marginalmartin}, but we omit these details here. 


\begin{astep} 
Write the joint association for $Y^n$ and $Y_{n+1}$ as
\begin{align}
\label{joint}
    T(Y^n) = a_{T}\left(V,\theta\right) \quad \text{and} \quad
    Y_{n+1} = a(U_{n+1},\theta), \quad (V, U_{n+1}) \sim \prob_{V,U_{n+1}},
\end{align}
where $a_T$ and $a$ are known functions, and $\prob_{V,U_{n+1}}$ is a known distribution. The left-most equation above represents a dimension-reduced association between data, parameter and auxiliary variable $V$, e.g., with $T(Y^n)$ a minimal sufficient statistic for $\theta$.  The key assumption is that the equation $T(Y^n) = a_T(V,\theta)$ can be solved for $\theta$; denote the solution by $\theta=\theta(T(Y^n),V)$.  Plugging this solution into the second equation gives 
\begin{equation}
\label{predassociation0}
Y_{n+1} = a\bigl( U_{n+1},\theta(T(Y^n),V) \bigr)
\end{equation}
as the marginal association for $Y_{n+1}$.  For fixed $Y^n$, the right-hand side is a function of random variables $(V, U_{n+1})$; write $G_{Y^n}$ for its distribution function.  Then \eqref{predassociation0} can be rewritten as $Y_{n+1} = G^{-1}_{Y^{n}}(W)$ with $W \sim \unif(0,1)$. 
 \end{astep}
 
\begin{pstep} 
Specify a valid random set $\S \sim \prob_\S$ that targets the unobservable auxiliary variable $W$ above, i.e., 
a random set whose contour function $\gamma_\S$ as in \eqref{eq:contourfunction} 
satisfies $\prob_W\{\gamma_\S(W) \leq \alpha\} \leq \alpha$ for all $\alpha \in [0,1]$.
\end{pstep}

\begin{cstep} 
Construct a data-dependent random set $G_{y^n}^{-1}(\S)$, the image of $\S$ under $G_{y^n}^{-1}$.  Then prediction of $Y_{n+1}$ is based on summaries of the distribution of $G_{y^n}^{-1}(\S)$. That is, if $A$ is some assertion about $Y_{n+1}$, then the plausibility of $A$, based on data $y^n$, is 
\begin{equation}
\label{plausprediction}
 \uPi_{y^n}(A) = \prob_{\S}\{G^{-1}_{y^n}(\S) \cap A \neq \varnothing\}.
\end{equation}
This can be evaluated for any $A$, resulting in a distribution-like summary of our uncertainty about $Y_{n+1}$, given $y^n$. 
As described in Section \ref{ss:cstep}, the contour 
\[ \pi(y_{n+1};y^n) = \prob_\S\{G_{y^n}^{-1}(\S) \ni \tilde y\}. \]
can be used to evaluate singleton assertions, i.e., $A=\{\tilde y\}$ for varying $\tilde y$, as well as generate plots that are useful for visualization; see Figures~\ref{fig:histogram}--\ref{fig:biplaus} above. Also, recall that when the random set $\S$ is nested, 
the plausibility is {\em consonant} 
and the function $A \mapsto\uPi_{y^n}(A)$ is completely determined by $\pi(\cdot;y^n)$ according to the rule 
\[ \uPi_{y^n}(A) = \sup_{\tilde y \in A} \pi(\tilde y;y^n). \]
 \end{cstep}

\citet{Martin2016PriorFreePP} show that, under certain conditions, the above construction leads to validity guarantees:
\begin{equation}
\label{eq:valid.parametric}
\sup_{\theta}\prob_{Y^{n+1}|\theta}\{\pi(Y_{n+1};Y^n)) \leq \alpha\}\leq \alpha, \quad \text{for all $\alpha \in [0,1]$}.
\end{equation}
That is, the contour value, $\pi(Y_{n+1};Y^n)$, as a function of the data $Y^{n+1}$ is stochastically no smaller than $\unif(0,1)$.  
This is desirable because we would be prone to errors in prediction if the actual $Y_{n+1}$ was determined to be relatively implausible based on data $Y^n$. If we define a $100(1-\alpha)$\% prediction plausibility region as in \eqref{eq:conformal.plint}, 
then \eqref{eq:valid.parametric} above ensures that the frequentist prediction coverage probability of this plausibility region is at least the nominal level, uniformly over $\theta$.

\subsection{Connections to existing literature}
\label{S:classical}

An important special case arises when the $Y_i$'s are scalar and, in the A-step of the nonparametric IM construction of Section~\ref{SSS:new.Astep}, one considers $T_i = Y_i$ in \eqref{eq:rule}, i.e.,
\begin{equation}
\label{eq:rule.special}
\psi_i(Y^{n+1}) = Y_i, \quad i \in \I_{n+1}.    
\end{equation}
The consequence is that the association in \eqref{eq:generalized.marginal} becomes
\[  Y_i = F^{-1}(U_i), \quad i \in \I_{n+1},  \]
where the $U_i$'s are exchangeable, marginally $\unif(0,1)$ and $F$ is the continuous distribution function that characterizes the common marginal distributions of the $Y_i$'s. As the dimension-reduced association in \eqref{eq:conformal.assoc} is now in terms of $r(Y_{n+1})$, the rank of $Y_{n+1}$, it can be rewritten as 
\begin{equation}
\label{eq:order}
Y_{n+1} \in [Y_{(V-1)}, Y_{(V)}] \quad V \sim \unif(\I_{n+1}),    
\end{equation}
where $Y_{(0)}$ and $Y_{(n+1)}$ are defined as the infimum and supremum of the support of the marginal distribution of $Y_1$, respectively.

Recall that the choice of the random set \eqref{eq:S} in the P-step of Section \ref{sss:p-step} was justified by the structure of the function $\psi_i$ as a non-conformity measure. There, we considered mappings $\psi_i$ such that smaller values of them are ``better'' in the sense that $y_i$ agrees with the prediction derived based on data $y_{-i}^{n+1}$. However, this is not necessarily the case for $\psi_i$ in \eqref{eq:rule.special}, as we are merely returning $y_i$ when comparing it to $y_{-i}^{n+1}$.  How to motivate a particular choice of random set $\S$?  \citet{Martin2016PriorFreePP} consider different classes of assertions/hypotheses about $Y_{n+1}$, and the structure of these assertions can inform the choice the random set.  For one-sided assertions, e.g., $\{Y_{n+1} \leq \tilde y\}$ or $\{Y_{n+1} \geq \tilde y\}$, for varying $\tilde y$, the natural choice of random set is similarly one-sided.
Here, however, we will focus on singleton assertions, $\{Y_{n+1} = \tilde y\}$, and the corresponding IM will produce two-sided prediction intervals.  Specifically, for the P-step, we recommend 
\begin{equation}\label{S3}
\S = \bigl\{\tfrac{n+2}{2} - |\tilde{V} - \tfrac{n+2}{2}|,\ldots,\tfrac{n+2}{2} + |\tilde{V} - \tfrac{n+2}{2}|\bigr\}, \quad \tilde V \sim \unif(\I_{n+1}), 
\end{equation}
which consists of a random number of points closest to the midpoint $\frac{n+2}{2}$.  This is the discrete version of what \cite{martinbook}, call the ``default'' random set.  Moreover, it is also easy to check that the random set $\S$ above is valid in the sense of \eqref{eq:S.valid}.

For the C-step, we combine $\S$ with the association \eqref{eq:order}, now represented as a collection of $v$-indexed sets
\begin{equation}
\label{eq:generalized.order}
\YY_{y^n}(v) = [y_{(v-1)}, y_{(v)}], 
\quad v \in \I_{n+1},
\end{equation}
to obtain the following new random set on the $Y_{n+1}$ space:
    \[ \YY_{y^n}(\S) =\{ y_{n+1}: y_{(\frac{n}{2} - |\tilde{V} - \frac{n+2}{2}|)} \leq y_{n+1} \leq y_{(\frac{n+2}{2} + |\tilde{V} - \frac{n+2}{2}|)}\}. \]
    The resulting plausibility contour is
\begin{align}\label{eq:twosided}
   \pi(y_{n+1};y^n) & = \prob_\S\bigl\{\YY_{y^n}(\S) \cap \{y_{n+1}\} \neq \varnothing \bigr\} \nonumber \\ 
   &= \prob_{\tilde{V}}\{y_{(\frac{n}{2} - |\tilde{V} - \frac{n+2}{2}|)} \leq y_{n+1} \leq y_{(\frac{n+2}{2} + |\tilde{V} - \frac{n+2}{2}|)}\} \nonumber \\ 
   &= \frac{1}{n+1} \sum_{v=1}^{n+1} 1_{y_{(\frac{n}{2} - |v - \frac{n+2}{2}|)} \leq y_{n+1} \leq y_{(\frac{n+2}{2} + |v - \frac{n+2}{2}|)}}.   
\end{align}

Using the same argument in the proof of Theorem~\ref{thm:valid}, it can be verified that the probabilistic predictor defined by the consonant plausibility contour in \eqref{eq:twosided} is Type-2 valid in the sense of Definition~\ref{def:strong}. In addition, the  prediction intervals derived from it are equal-tailed, always including the median of the observed data $y^n$; see Section~\ref{ss:one}, Figure~\ref{fig:histogram}(b). 
The reader may recognize that these prediction intervals are not new, they agree with the classical intervals based on order statistics where, for pre-specified integers $r$ and $s$ such that $1 \leq r < s \leq n$, the interval 
\begin{equation}\label{eq:classic}
[Y_{(r)},Y_{(s)}]
\end{equation}
is a $100\frac{s-r}{n+1}\%$ prediction interval \citep{wilks1941}.
It is satisfying to see that the IM framework, which has so far been focused on situations with a parametric statistical model, can be extended to cases without such a model and provide what would be considered a classical solution. Our analysis also sheds new light on that classical solution, in particular, it reveals that the latter has an (imprecise) ``probabilistic'' interpretation, beyond its more familiar interpretation as a frequentist procedure for constructing prediction intervals.  


Readers familiar with the imprecise probability literature are sure to recognize some connections here to the {\em nonparametric predictive inference} presented in, e.g., \citet{AUGUSTIN2004251} and elsewhere.  Based on what is referred to in this literature as {\em Hill's assumption} \citep{hill1968, HILL1993}, a pair of lower and upper probabilities for $Y_{n+1}$, given $Y^n = y^n$, are defined as, respectively, 
\begin{align*}
\underline{\prob}(Y_{n+1} \in B \mid y^n) & = \frac{1}{n+1} \sum_{v=1}^{n+1} 1_{\YY_{y^n}(v) \subseteq B} \\
\overline{\prob}(Y_{n+1} \in B \mid y^n) & = \frac{1}{n+1} \sum_{v=1}^{n+1}1_{\YY_{y^n}(v) \cap B \neq \varnothing}, 
\end{align*}
where $B$ is some generic set in the support of $Y_{n+1}$ and $\YY_{y^n}(v)$, for $v \in \I_{n+1}$, is as in \eqref{eq:generalized.order}.  
\citet{CoolenBayes} interprets these in a ``best of both worlds'' sense: on one hand, these are genuine post-data lower and upper probabilities and, therefore, inherit certain coherence properties \citep[e.g.,][]{walley1991}; on the other hand, through their connection to the underlying data-generating process through Hill's assumption, they inherit certain calibration properties \citep[e.g.,][]{LawlessandFredette2005}.  It turns out that these lower and upper probabilities can also be obtained from the above IM construction.  Indeed, if we take $\S = \{\tilde V\}$, for $\tilde V \sim \unif(\I_{n+1})$, then the induced distribution of the data-dependent random set $\YY_{y^n}(\S)$ generates the lower and upper probabilities defined above.  Although the aforementioned lower and upper probability approach is a special case of the IM construction presented above, there are some key differences between the two, resulting from our proposed random set $\S$ being nested and, therefore, structurally different from the non-nested singleton random sets that generate the $(\underline{\prob}, \overline{\prob})$ output.  In particular, there is no useful sense in which once can assess the ``plausibility'' that the next observation, $Y_{n+1}$, will be equal to some specified $y$: the lower probability equals 0 and the upper probability equals $(n+1)^{-1}$ for every $y$.  Compare this to the plausibility contour displayed in Figure~\ref{fig:histogram}(b), where there is one $y$ which is the ``most plausible'' value of $Y_{n+1}$ based on the observed data $y^n$.  Having this plausibility contour makes reading off prediction intervals with desired coverage level straightforward.  Of course, this same calibration property is embedded in the lower and upper probabilities above, but it is more difficult to extract when represented in that form.  Specifically, given a lower probability like that above, a nominal $100(1-\alpha)$\% prediction region can be obtained by finding the smallest $B$ such that $\underline{\prob}(Y_{n+1} \in B \mid y^n)$ exceeds $1-\alpha$, i.e., 
\[ \bigcap \bigl\{B \subseteq \RR: \underline{\prob}(Y_{n+1} \in B \mid y^n) \geq 1-\alpha \bigr\}. \]
In this case, with $\underline{\prob}$ as defined above, this calculation returns the classical prediction interval solution given in \eqref{eq:classic}. 

\bibliographystyle{apalike}
\bibliography{literature}

\end{document}